\documentclass[11pt]{article}
\usepackage{amsmath,amsthm, amssymb, amsfonts,accents,nccmath}
\usepackage{enumitem}
\usepackage{array}
\usepackage[toc,page]{appendix}
\usepackage[english]{babel}
\usepackage{dsfont}
\usepackage{booktabs}
\usepackage[linesnumbered,commentsnumbered,ruled,vlined]{algorithm2e}
\usepackage{tikz}
\usepackage{tikz-network}
\usepackage[utf8]{inputenc}
\usepackage[english]{babel}

\usepackage{caption}
\usepackage{subcaption}
 
\usepackage{hyperref}
\hypersetup{
    colorlinks=true,
    linkcolor=blue,
    filecolor=darkmagenta,      
    urlcolor=cyan,
    citecolor=red
} 
\makeatletter
\def\BState{\State\hskip-\ALG@thistlm}
\makeatother
\numberwithin{equation}{section}

\newtheorem{theorem}{Theorem}[section]
\newtheorem{cor}[theorem]{Corollary}
\newtheorem{lem}[theorem]{Lemma}
\newtheorem{prop}[theorem]{Proposition}

\newtheorem{rem}[theorem]{Remark}
\newtheorem{ex}[theorem]{Example}

\newtheorem{assume}[theorem]{Assumption}

\usepackage{mathtools}
\DeclarePairedDelimiter\ceil{\lceil}{\rceil}
\DeclarePairedDelimiter\floor{\lfloor}{\rfloor}

\newcommand{\ds}{\displaystyle}

\title{On Squared Distance Matrix of Complete Multipartite Graphs}
\author{Joyentanuj Das\footnote{Department of Applied Mathematics,
National Sun Yat-sen University, Gushan District, Kaohsiung City, 804, Taiwan (ROC) \indent  Email: joyentanuj@gmail.com,  joyentanuj@math.nsysu.edu.tw}  \quad and \quad Sumit Mohanty\footnote{Humanities and Applied Sciences, IIM Ranchi,  Suchana Bhawan, Audrey House Campus, Meur's Road, Ranchi, Jharkhand-834008, India. \indent   Email:  sumitmath@gmail.com, sumit.mohanty@iimranchi.ac.in}}

\date{}

\begin{document}

\maketitle

\begin{abstract}
 Let $G = K_{n_1,n_2,\cdots,n_t}$ be a complete $t$-partite graph  on $n=\sum_{i=1}^t n_i$ vertices. The distance between vertices $i$ and $j$ in $G$, denoted by  $d_{ij}$   is defined to be  the length of the shortest path between $i$ and $j$. The squared distance matrix $\Delta(G)$ of $G$ is the $n\times n$ matrix with $(i,j)^{th}$ entry equal to $0$ if $i = j$ and equal to $d_{ij}^2$ if $i \neq j$.    We define the squared distance energy $E_{\Delta}(G)$ of $G$ to be the sum of the absolute values of its eigenvalues. We determine the inertia of $\Delta(G)$ and compute the squared distance energy $E_{\Delta}(G)$. More precisely, we prove that  if $n_i \geq 2$ for $1\leq i \leq t$, then  $ E_{\Delta}(G)=8(n-t)$  and if  $ h= |\{i : n_i=1\}|\geq 1$, then  $$ 8(n-t)+2(h-1) \leq E_{\Delta}(G) < 8(n-t)+2h.$$
Furthermore, we show that for a fixed value of $n$ and $t$, both the spectral radius of the squared distance matrix  and the squared distance energy of  complete $t$-partite graphs on $n$ vertices are maximal for complete split graph $S_{n,t}$ and  minimal for Tur{\'a}n graph $T_{n,t}$.
\end{abstract}


\section{Introduction  and Motivation}
A simple connected graph $G$  is a metric space with respect to the metric $d,$ where $d(i,j)$ equals the length of the shortest path between the vertices $i$ and $j$. We set $d(i,i)=0$ for every vertex $i$ in $G$.   The distance matrix of  a graph  $G$ on $n$ vertices  is an  $n \times n$ matrix  $D(G) = [d_{ij}]$, where $d_{ij}= d(i,j)$. The squared distance matrix  $\Delta(G)$ of $G$ is defined to be the Hadamard product $D(G) \circ D(G)$, and thus   the squared distance matrix of $G$  is an  $n \times n$ matrix, defined as $\Delta(G)=[d_{ij}^2].$ Hence, $\Delta(G)$  is a real symmetric  matrix and  the eigenvalues of $\Delta(G)$  are real.  Let $\lambda_1(G)\geq \lambda_2(G)\geq \cdots \geq  \lambda_n(G)$ be the eigenvalues of $\Delta(G)$ and   $\sigma(\Delta(G))=\{\lambda_1(G)\geq \lambda_2(G)\geq \cdots \geq \lambda_n(G) \}$  denote the spectrum of $\Delta(G)$. The spectral radius of $\Delta(G)$, denoted by $\rho(G)$ is defined as $\ds \rho(G)=\max_{\lambda \in \sigma(\Delta(G))}|\lambda|$ and by Perron-Frobenius theory,  the spectral radius $\rho(G)$ is the largest eigenvalue of $\Delta(G)$, {\it i.e.,} $\rho(G)= \lambda_1(G)$. We define the squared distance energy of  $\Delta(G)$  as $E_{\Delta}(G)= \sum_{i=1}^n |\lambda_i(G)|$. Since the trace of $\Delta(G)$ is zero, so
\begin{equation*}
 E_{\Delta}(G)= 2 \sum_{ \lambda_i(G) <0} |\lambda_i(G)|.
\end{equation*} 

Let $T$ be a tree with $n$ vertices. In~\cite{Gr1}, the authors proved that the determinant of the distance matrix $D(T)$ of $T$ is given by $\det D(T)=(-1)^{n-1}(n-1)2^{n-2}.$  Note that, the determinant does not depend on the structure of the tree but the number of vertices. In \cite{Gr2}, it was shown that the inverse of the distance matrix of tree $T$ is given by $D(T)^{-1} = -\dfrac{1}{2}L(T) + \dfrac{1}{2(n-1)}\tau \tau^t,$ where $L(T)$ is Laplacian matrix of $T$ and $\tau = (2-\delta_1,2-\delta_2,\cdots,2-\delta_n)^t,$  $\delta_i$ denotes the degree of the vertex $i$. Several extensions and generalizations of these results has been proved (see, for example see~\cite{Bp3,Gr3,Hou1,Hou2}). In~\cite{Bp4},  authors computed the determinant of the squared distance matrix $\Delta(T)$ of a tree $T$  with no vertex of degree $2$ and later in~\cite{Bp5},  find its inverse as a rank one perturbation of a Laplacian-like matrix. In~\cite{Bp6}, the author considered the squared distance matrix of weighted trees and under certain conditions found its inverse as a rank one perturbation of a Laplacian-like matrix. However, the results for distance matrix and squared distance matrix for trees are similar,  but strategy adopted for the proofs are different due to the non-linearity of the Hadamard product. For the definition of Laplacian-like matrix and other related results see~\cite{JD1, Zhou1}. Moreover, in~\cite{Bp5}, authors also determined the inertia of the squared distance matrix of a tree. These developments encourage us to study problems related to the squared distance matrix on other class of graphs. More precisely, we consider a few problems on the squared distance matrix of  complete multipartite graphs.

Before proceeding further, we first introduce a few notations which will be used time and again throughout this article. We write $G\approxeq H$ to indicate that two graphs $G$ and $H$ are isomorphic. Let $I_n$  denote the identity matrix,  $J_{m \times n}$  denotes the $m\times n$ matrix of all ones and if $m=n$, we use the notation $J_m$.  We write $\mathbf{0}_{m \times n}$ to represent zero matrix of order $m \times n$ and simply write $\mathbf{0}$ if there is no scope of confusion with the order of the matrix.   For a  symmetric  matrix $A$, the inertia of  $A$, denoted by $\textup{In}(A)$ is  the triplet  $(\mathbf{n}_{+}(A),\mathbf{n}_{0}(A), \mathbf{n}_{-}(A) )$, where $\mathbf{n}_{+}(A),\mathbf{n}_{0}(A)  \mbox{ and } \mathbf{n}_{-}(A)$ denote the number of positive eigenvalues of $A$, the multiplicity of  $0$  as an eigenvalue of $A$ and the number of negative eigenvalues of $A$, respectively. 

For $n\geq 2$, let the $i^{th}$ largest component in a vector $X =(x_1,x_2,\cdots, x_n)\in \mathbb{R}^n$ be denoted by $x_{[i]}$ and let $X_{\downarrow}=( x_{[1]}, x_{[2]}, \cdots ,x_{[n]})$ denote the decreasing rearrangement of $x$. For $X =(x_1,x_2,\cdots, x_n)$ and $Y =(y_1,y_2,\cdots, y_n)$, we say that $X$ majorizes $Y$, and write $X \succeq Y$ if $$\sum_{i=1}^k x_{[i]} \geq \sum_{i=1}^k y_{[i]} \mbox{ for } k=1,2,\ldots, n,$$ with equality for $k=n.$ We write $X\succ Y$ if  $X\succeq Y$ and $X_{\downarrow}\neq Y_{\downarrow}$.

For $t\geq 2$, let $K_{n_1,n_2,\cdots,n_t}$ denote the complete multipartite graph on $n=\sum_{i=1}^t n_i$ vertices. The complete graph $K_t$ ($t\geq 2$) can also be seen as  complete $t$-partite graph  $K_{n_1,n_2,\cdots,n_t}$ with $n_i=1$ for all $1\leq i\leq t$. Let us  recall the definition of two special types of complete $t$-partite graphs on $n$ vertices, namely complete split graph $S_{n,t}$ and Tur{\'a}n graph $T_{n,t}$:
\begin{enumerate}
\item Complete split graph $S_{n,t} \approxeq K_{n-t+1,\scriptsize{\underbrace{1,1,\ldots,1}_{t-1\ times}}}$ consisting of an independent set of $n-t+1$ vertices and a clique of $t-1$ vertices, such that each vertex of the  independent set is adjacent  to each  vertex of the clique.

\item Tur{\'a}n graph $T_{n,t} \approxeq K_{\ceil*{n/t}, \cdots, \ceil*{n/t},\floor*{n/t},\cdots, \floor*{n/t}}$, the $(t+1)$-clique-free graph with maximum number of edges. 
\end{enumerate}

In~\cite{Indulal}, the distance energy of a graph is defined as the sum of the absolute values of the distance eigenvalues.  Subsequently, many interesting results were obtained related to distance energy (for example see~\cite{Sun,So,Stev,Var,Zhang-x} and~\cite[Section 8]{Aouchiche}). In particular, the distance energy of complete $t$-partite graph  has been studied in~\cite{So,Stev,Zhang-x}. More precisely, in~\cite{Zhang-x}, the author  obtained the inertia of the distance matrix of $K_{n_1,n_2,\cdots,n_t}$ and also consider the extremization (maximization and minimization) problem for the distance energy of $K_{n_1,n_2,\cdots,n_t}$.  

In this manuscript,   we find the  inertia of the squared distance matrix $\Delta(K_{n_1,n_2,\cdots,n_t})$ and using the  inertia of $\Delta(K_{n_1,n_2,\cdots,n_t})$, we compute  the squared distance energy   $E_{\Delta}(K_{n_1,n_2,\cdots,n_t})$. Further, we  prove the following: For fixed value of $n$ and $t$, let $n=\sum_{i=1}^t n_t$.  Then  $ E_{\Delta}(K_{n_1,n_2,\cdots,n_t})=8(n-t)$ if $n_i \geq 2$ for $1\leq i \leq t$  and  $$ 8(n-t)+2(h-1) \leq E_{\Delta}(K_{n_1,n_2,\cdots,n_t}) < 8(n-t)+2h \mbox{ if } h= |\{i : n_i=1\}|.$$ 
Furthermore, we show that if $ (m_1,m_2,\cdots,m_t)\succ (n_1,n_2,\cdots,n_t)$, then 
$$E_{\Delta}(K_{m_1,m_2,\cdots,m_t}) \geq E_{\Delta}(K_{n_1,n_2,\cdots,n_t}) \mbox{ and }\rho(K_{m_1,m_2,\cdots,m_t}) > \rho(K_{n_1,n_2,\cdots,n_t}).$$
As a consequence, we can conclude that   $S_{n,t}$ and  $T_{n,t}$ are extremal graphs for which the squared distance energy and as well as the spectral radius of the squared distance matrix of complete $t$-partite graphs on $n$  vertices attain its maxima and minima, respectively. Moreover,  we found that in the case of the spectral radius, these extremal graphs are unique up to isomorphism,  but for the squared distance energy, the maximal graph $S_{n,t}$ is unique up to isomorphism, whereas the minimal graph $T_{n,t}$ is unique up to isomorphism if and only if $n\leq 2t+1$. 

For fixed value of $n$, $t$ and $h$, let $\mathcal{M}(n,t,h)$ be the class of complete $t$-partite graphs on $n$ vertices where the number of partitions of size $1$ is $h$,
 {\it i.e,} 
$$
\mathcal{M}(n,t,h)=\left\{K_{n_1,n_2,\cdots,n_t} \; : \; n=\sum_{i=1}^t n_i \mbox{ and }h=|\{ i: n_i=1\}| \right\}.$$
In our investigation, we observe that  the major distinction is in the spectral properties of  $\Delta(K_{n_1,n_2,\cdots,n_t})$ and $D(K_{n_1,n_2,\cdots,n_t})$  if $K_{n_1,n_2,\cdots,n_t}$ in $\mathcal{M}(n,t,h)$. For example,  unlike the case of $D(K_{n_1,n_2,\cdots,n_t})$, if $0$ is an eigenvalue of $\Delta(K_{n_1,n_2,\cdots,n_t})$, then it is of multiplicity one.

This article is organized as follows. In Section~\ref{sec:element}, we discuss a few  preliminary results useful for this article. In Section~\ref{sec:char-poly}, we compute the characteristic polynomial of $\Delta(K_{n_1,n_2,\cdots,n_t})$ and also find some of its related results. In Section~\ref{sec:inertia-n-energy}, we determine the inertia of $\Delta(K_{n_1,n_2,\cdots,n_t})$ and compute the squared distance energy  $E_{\Delta}(K_{n_1,n_2,\cdots,n_t})$. Next,  in Section~\ref{sec:max-min-energy}, we consider problem related to maximization and minimization of  $E_{\Delta}(K_{n_1,n_2,\cdots,n_t})$ subject to different conditions and use majorization techniques to find  the  extremal graphs for which $E_{\Delta}(K_{n_1,n_2,\cdots,n_t})$ attains its the maximum  and minimum. In this case, we also   obtain the requisite conditions under which these extremal graphs are unique upto isomorphism. Finally, in Section~\ref{sec:sp-radius}, using majorization techniques we show  that for fixed value of $n$ and $t$, if $n=\sum_{i=1}^t n_i$,  then  the spectral radius $\rho(K_{n_1,n_2,\cdots,n_t})$ of the squared distance matrix $\Delta(K_{n_1,n_2,\cdots,n_t})$  uniquely attains   its maximum value at $S_{n,t}$ and minimum value at $T_{n,t}$. We also provide an alternative proof for the same.

\section{Some Preliminary Results}\label{sec:element}

In this section, we discuss a few preliminary results useful for our subsequent sections. We first recall some existing results on majorization and matrix theory.

Let $\mathcal{B}_n = \{(z_1,z_2,\cdots,z_n) \in \mathbb{Z}^n : z_1 \ge z_2 \ge \cdots \ge z_n\}$. For $1\leq j \leq n,$ let $E^n_{j}$ be the $n$-tuple  with $1$ at the $j^{th}$ entry and for  $1 \le j,k \le n$, we use $E_{j,k}^n$ to denote $E_k^n-E_j^n$. With these notations  we now state the following results.

\begin{theorem}\cite{Ob2}\label{thm:maj_seq}
Let $X, Y \in \mathcal{B}_n$ and $X \neq Y$. Then $Y \succ X$ if and only if there exists a sequence in $\mathcal{B}_n$ such as $Y = Y_0 \succ Y_1 \succ \cdots \succ Y_l = X$ such that for 
$1\leq i \leq l, Y_i = Y_{i-1} + E_{j_i,k_i}^n$ for some $k_i > j_i$.
\end{theorem}

\begin{theorem}\cite{Ob}\label{thm:maj_tuple}
Let $n_1, m_1, \cdots , n_k, m_k$ and $x_1, \cdots , x_t$ be  real numbers. Then
$$(n_1,n_2, \cdots , n_k) \succeq (m_1,m_2, \cdots , m_k) \textup{ if and only if } (n_1, \cdots , n_k, x_1, \cdots , x_t) \succeq (m_1, \cdots, m_k, x_1, \cdots , x_t).$$
\end{theorem}

For every integer $t \ge 1$, let $\mathcal{C}_t = \{(n_1,n_2,\cdots,n_t) \in \mathbb{Z}^t : n_1 \ge n_2 \ge \cdots \ge n_t \ge 1\}$. Note that $(\mathcal{C}_t, \succeq)$ is a partially ordered set.

\begin{theorem}\cite{Ob2}\label{thm:thm_compare}
Let $(\mathcal{C}_t, \succeq)$ be the partially ordered set defined as above. Then for every $(n_1,n_2,\cdots,n_t) \in \mathcal{C}_t$  with $n = n_1+\cdots+n_t,$ we have 
$$
(n-t+1, \underbrace{1,\cdots,1}_{t-1}) \succeq (n_1,n_2,\cdots,n_t) \succeq \Big(\ceil*{\frac{n}{t}},\cdots,\ceil*{\frac{n}{t}},\floor*{\frac{n}{t}},\cdots,\floor*{\frac{n}{t}}\Big).
$$ 
\end{theorem}

The result below follows from  Theorem~\ref{thm:thm_compare}, and hence presented as a corollary without proof.


\begin{cor}\label{lem:major222}
For $t\geq 1$ and   let  $(\mathcal{C}_t', \succeq)$ be a partially ordered set, where $\mathcal{C}_t' = \{(n_1,n_2,\cdots,n_t) \in \mathbb{Z}^t : n_1 \ge n_2 \ge \cdots \ge n_t \ge 2\}$. Then for every $(n_1,n_2,\cdots,n_t) \in \mathcal{C}_t'$  with $n = \sum_{i=1}^t n_i,$ we have 
$$
(n-2(t-1), \underbrace{2,2,\cdots,2}_{t-1}) \succeq (n_1,n_2,\cdots,n_t).
$$ 
\end{cor}

Now we state a well known result namely inclusion principle that gives an interlacing relation between the eigenvalues of a real symmetric matrix and its principal minors.
\begin{theorem}\cite[Theorem 4.3.28]{Horn}\label{thm:inter}
Let $A$ be an $n\times n$ real symmetric matrix, let $r$ be an integer with $1\leq r\leq n$, and let $B$ denote any $r \times r$ principal submatrix of $A$. Let $\lambda_1(A) \geq \lambda_2(A)\geq \cdots \geq   \lambda_{n}(A) $ and $\lambda_1(B) \geq \lambda_2(B)\geq \cdots \geq   \lambda_{r}(B) $ be the eigenvalues of $A$ and $B$, respectively. For each integer $i$ such that $1\leq i\leq r$, we have
$$ \lambda_i(A)\geq \lambda_i(B) \geq \lambda_{n-r+i}(A).$$
\end{theorem}

The next  result is a consequence of the inclusion principle and Weyl's inequality that gives  interlacing inequalities of a rank one perturbation to a real symmetric matrix.
\begin{lem}\cite[Corollary 4.3.9]{Horn}\label{lem:rank-pert-ev}
Let $A$ be an $n\times n$ real symmetric matrix, and $B=A+J_n$, where $J_n$ is the matrix of all ones. Let $\lambda_1(A) \geq \lambda_2(A)\geq \cdots \geq \lambda_{n-1}(A)\geq  \lambda_{n}(A)$ and $\lambda_1(B) \geq \lambda_2(B)\geq \cdots \geq \lambda_{n-1}(B)\geq \lambda_{n}(B)$  be the eigenvalues of $A$ and $B$, respectively. Then the following interlacing inequalities holds:
$$\lambda_1(B)\geq \lambda_2(A) \geq \lambda_2(B)\geq \cdots \geq \lambda_{n-1}(A)\geq  \lambda_{n}(B)\geq \lambda_{n}(A).$$
\end{lem}

Observe that, given a monic polynomial $f$ with real coefficients we know that $f(x)$ is positive for sufficiently large value $x$.  The lemma below is an immediate consequence of this observation.

\begin{lem}\label{lem:comp-r-root}
Let $f$ and $g$ be   monic polynomials with real coefficients having real roots. Let $\lambda_1(f)$ and  $\lambda_1(g)$ denote the largest real root of $f$ and $g$, respectively. If $f(\lambda_1(g))<0$, then $\lambda_1(g) < \lambda_1(f).$
\end{lem}

We conclude this section with a lemma that gives two inequalities. These inequalities can be obtained with simple calculations, and hence we omit the proofs. 

\begin{lem}\label{lem:ab-ineq}
Let  $n_p,n_q $ be real numbers. Then the following inequalities hold.
\begin{enumerate}
\item[($i$)] If $a>0$ and $b>3$, then 
$$ \frac{n_p}{a}+ \frac{n_q}{b} > \frac{n_p-1}{a+3}+ \frac{n_q+1}{b-3} \textup{\ if and only if } (3n_p+a)b(b-3) > (3n_q+b)a(a+3).$$

\item[($ii$)] If $a>3$ and $b>0$, then 
$$ \frac{n_p}{a}+ \frac{n_q}{b} > \frac{n_p-1}{a-3}+ \frac{n_q+1}{b+3} \textup{\ if and only if } 3[n_q a(a-3)-n_p b(b+3)] > ab(a-b-6).$$
\end{enumerate}
\end{lem}

In the next section, we compute the characteristic polynomial of the squared distance matrix of the complete $t$-partite graph $K_{n_1,n_2,\cdots,n_t}$ and discusses some of its consequences.

\section{Characteristic Polynomial of $\Delta(K_{n_1,n_2,\cdots,n_t})$}\label{sec:char-poly}
Let $\Delta(K_{n_1,n_2,\cdots,n_t})$ be the squared distance matrix of the complete $t$-partite graph $K_{n_1,n_2,\cdots,n_t}$. Thus, $\Delta (K_{n_1,n_2,\cdots,n_t})$ can be expressed in the following block form 
\begin{small}
\begin{equation}\label{eqn:D(K)}
\Delta(K_{n_1,n_2,\cdots,n_t}) = \left[\begin{array}{c|c|c|c}
4(J_{n_1} - I_{n_1}) & J_{n_1 \times n_2} & \cdots & J_{n_1 \times n_t}\\
\hline
J_{n_2 \times n_1} &4(J_{n_2} - I_{n_2}) & \cdots & J_{n_2 \times n_t}\\
\hline
\vdots & \cdots  & \ddots &\vdots\\
\hline
J_{n_t \times n_1} & J_{n_t \times n_2} &  \cdots & 4(J_{n_t} - I_{n_t})\\
\end{array} \right].
\end{equation}
\end{small}

 We begin with a few lemmas which will be  used to compute the characteristic polynomial of $\Delta(K_{n_1,n_2,\cdots,n_t})$.
 
 \begin{lem}\label{Lem:Det_1}
Let $x \in \mathbb{R}$ and $A_m$ be an $m \times m$ matrix  of the following form
\begin{small}
\begin{equation*}\label{eqn:M_1}
A_m= \left[\begin{array}{cccc}
p_1 & p_1 & \cdots & p_1\\
p_2 & 4(p_2 - 1)-x & \cdots & p_2\\
\vdots & \cdots  & \ddots &\vdots\\
p_m  & p_m &  \cdots & 4(p_m - 1)-x\\
\end{array} \right].
\end{equation*} 
\end{small}
Then, the determinant of  $A_m$ is given by $$\det A_m = p_1\prod_{j \neq 1}(3p_j - 4-x).$$
\end{lem}

\begin{proof}
 By subtracting the first columns from all the other columns, the resulting matrix is of the following form
\begin{small}
$$
\left[\begin{array}{cccc}
p_1 & 0 & \cdots & 0\\
p_2 & 3p_2 - 4-x & \cdots & 0\\
\vdots & \vdots  & \ddots &\vdots\\
p_m  & 0 &  \cdots & 3p_m - 4-x\\
\end{array} \right]
$$
\end{small}
 and hence the result follows.
\end{proof}

\begin{lem}\label{Lem:Det_2}
Let $B_m$ be an $m \times m$ matrix  of the following form
\begin{small}
\begin{equation}\label{eqn:M_2}
B_m = \left[\begin{array}{cccc}
4(n_1 - 1) & n_1 & \cdots & n_1\\
n_2 & 4(n_2 - 1) & \cdots & n_2\\
\vdots & \vdots  & \ddots &\vdots\\
n_m  & n_m &  \cdots & 4(n_m - 1)\\
\end{array} \right].
\end{equation}
\end{small}
Then, 
$$\det (B_m-xI_m) = \sum_{i=1}^m \left( n_i\prod_{j \neq i}(3n_j - 4-x) \right) +  \prod_{i=1}^m (3n_i - 4-x).$$
\end{lem}

\begin{proof}
We prove the lemma using induction on the order of the matrix. For $m=1$, the result is true. Let us assume the result is true for matrix order $m-1$ of similar form. Now expanding along the first row and using Lemma~\ref{Lem:Det_1}, we get $$\det (B_m-xI_m) = 4(n_1-1) \det (B_{m-1}-xI_{m-1}) -n_1 \sum_{i=2}^{m} \left( n_i\prod_{j \neq i}(3n_j - 4-x) \right)$$ and hence the result follows.
\end{proof}

\begin{prop}\label{prop:char_eqn}
Let $G$ be a complete $t$-partite graph $K_{n_1,n_2,\cdots,n_t}$ on $n=\sum_{i=1}^t n_i$ vertices and $\Delta(G)$ be the squared distance matrix of $G$.  Then, the characteristics polynomial of $\Delta(G)$ is given by
\begin{equation}\label{eqn:char_eqn}
P_{\Delta}(G,x)=(x+4)^{n-t}\left[  \prod_{i=1}^t (x+4-3n_i) -\sum_{i=1}^t \Bigg( n_i\prod_{j \neq i}(x+4-3n_j) \Bigg) \right].
\end{equation}
\end{prop}
\begin{proof}
Let $\widetilde{\Delta}(G)= \Delta(G) -xI_n$. We use elementary row and column operations to prove the result. For each partition of the vertex set we subtract the first column from the remaining columns and then add all the rows of each partition to the first row. Next, we move the first column of each partition to the first $t$-columns followed by the same operation on the rows. Then, the resulting matrix has the following block form $$
\left[
\begin{array}{c|c}
B_t - xI_t& \mathbf{0} \\
\hline
* & -(x+4)I_{n-t}
\end{array}
\right],
$$ where $B_t$ is the matrix as defined in Eqn~(\ref{eqn:M_2}).
Thus,   $$\det(\Delta(G) -xI)=(-1)^{n-t} (x+4)^{n-t} \det(B_t - xI_t)$$ and from Lemma~\ref{Lem:Det_2}, we have
{\small
\begin{align*}
\det(\Delta(G) -xI_n)&= (-1)^{n-t} (x+4)^{n-t} \left[\sum_{i=1}^t \Bigg( n_i \prod_{j \neq i}(3n_j-4-x)\Bigg) + \prod_{i=1}^t (3n_i-4-x) \right]\\
                   &=(-1)^{n-t} (x+4)^{n-t} \left[ (-1)^{t-1}\sum_{i=1}^t \Bigg( n_i \prod_{j \neq i}(x+4-3n_j)\Bigg) + (-1)^{t}\prod_{i=1}^t (x+4-3n_i) \right]\\
                   &=(-1)^{n} (x+4)^{n-t} \left[  \prod_{i=1}^t (x+4-3n_i) -\sum_{i=1}^t \Bigg( n_i\prod_{j \neq i}(x+4-3n_j) \Bigg) \right],
\end{align*}}
and hence the result follows.
\end{proof}
\begin{rem}\label{rem:monic-char1}
 From  the proof of Proposition~\ref{prop:char_eqn}, it is easy to see that  
$$\det(xI_t - B_t)=  \prod_{i=1}^t (x+4-3n_i) -  \sum_{i=1}^t \Bigg( n_i\prod_{j \neq i}(x+4-3n_j)\Bigg).$$
\end{rem}
\begin{cor}\label{cor:detD_single}
Let $G$ be a complete $t$-partite graph $K_{n_1,n_2,\cdots,n_t}$ on $n=\sum_{i=1}^t n_i$ vertices and $\Delta(G)$ be the squared distance matrix of $G$.  Then, the determinant of $\Delta(G)$ is given by
$$\det \Delta(G) = (-4)^{n-t} \left[  \sum_{i=1}^t \Bigg( n_i\prod_{j \neq i}(3n_j-4) \Bigg) + \prod_{i=1}^t (3n_i-4) \right].$$
\end{cor}

\begin{cor}\label{cor:char-reduce}
Let $t=s+h\geq 2$ and   $n_1,n_2,\cdots,n_t$ be positive integers  such that $n_i \geq 2$ for $1 \leq i\leq s$  and $n_i=1$ for $s+1 \leq i\leq t=s+h$. Let $\Delta(G)$ be the squared distance matrix of the complete $t$-partite graph $G=K_{n_1,n_2,\cdots,n_t}$ on $n=\sum_{i=1}^t n_i$ vertices. Then, the characteristic polynomial of $\Delta(G)$ is given by $P_{\Delta}(G,x)=(x+4)^{n-t}(x+1)^{h-1} p(G,x),$ where
\begin{align}\label{eqn:char-reduce-monic1}
p(G,x)&= (x+1-h) \prod_{i=1}^s (x+4-3n_i)-(x+1)\sum_{i=1}^s \Bigg( n_i\prod_{j \neq i}(x+4-3n_j)  \Bigg)\\
&= (x+1) \det(xI-B_s)- h  \prod_{i=1}^s (x+4-3n_i),\nonumber 
\end{align}
where $B_s$ is the matrix  defined in Eqn~(\ref{eqn:M_2}) with parameters $n_1, n_2,\cdots, n_s$.
\end{cor}
\begin{proof}
Substituting  $n_i=1$ for $s+1 \leq i\leq t=s+h$ in  Eqn.~\eqref{eqn:char_eqn},  the characteristic polynomial $P_{\Delta}(G,x)$ of $\Delta(G)$ reduces to
{\small
\begin{align}\label{eqn:ch-poly-1}
P_{\Delta}(G,x)&=(x+4)^{n-t}\left[  \prod_{i=1}^t (x+4-3n_i) -\sum_{i=1}^t \Bigg( n_i\prod_{j \neq i}(x+4-3n_j) \Bigg) \right] \nonumber\\
&= (x+4)^{n-t} \Bigg[ (x+1)^h \prod_{i=1}^s (x+4-3n_i) - (x+1)^h \sum_{i=1}^s \Bigg( n_i\prod_{j \neq i}(x+4-3n_j) \Bigg)  \nonumber\\
& \hspace*{9.05cm} \left. - h (x+1)^{h-1} \prod_{i=1}^s (x+4-3n_i)  \right]  \nonumber\\
&= (x+4)^{n-t} (x+1)^{h-1} \Bigg[ (x+1) \left( \prod_{i=1}^s (x+4-3n_i) -  \sum_{i=1}^s \Bigg( n_i\prod_{j \neq i}(x+4-3n_j)  \Bigg) \right)  \nonumber\\
& \hspace*{10.7cm} \left. - h  \prod_{i=1}^s (x+4-3n_i)  \right]\\
&=(x+4)^{n-t}(x+1)^{h-1}\left[(x+1-h) \prod_{i=1}^s (x+4-3n_i)-(x+1)\sum_{i=1}^s \Bigg( n_i\prod_{j \neq i}(x+4-3n_j)  \Bigg) \right]. \nonumber
\end{align}}
Further, in view of Remark~\ref{rem:monic-char1},   Eqn.~\eqref{eqn:ch-poly-1} can be written as 
$$P_{\Delta}(G,x)=(x+4)^{n-t} (x+1)^{h-1} \Bigg[ (x+1) \det(xI-B_s)- h  \prod_{i=1}^s (x+4-3n_i) \Bigg].$$
This completes the proof.
\end{proof}

\section{Inertia and Energy of $\Delta(K_{n_1,n_2,\cdots,n_t})$}\label{sec:inertia-n-energy}
In this section, we  find the inertia of the squared  distance matrix of the complete multipartite graph $K_{n_1,n_2,\cdots,n_t}$ and compute the squared  distance energy $E_{\Delta}(K_{n_1,n_2,\cdots,n_t})$. We  first consider the complete multipartite graph $K_{n_1,n_2,\cdots,n_t}$ whenever $n_i\geq 2$ for $1\leq i\leq t$.

\begin{theorem}\label{thm:ni-geq-2}
Let $t\geq 2$ and $n_1,n_2,\cdots,n_t$ be positive integers  such that $n_i \geq 2$ for $1 \leq i\leq t$. 
Let $G$ be the complete $t$-partite graph $K_{n_1,n_2,\cdots,n_t}$ on $n=\sum_{i=1}^t n_i$ vertices and $\Delta(G)$ be the squared distance matrix of $G$.  Then, the inertia of $\Delta(G)$ is given by
$$\textup{In}(\Delta(G))=(t, 0, n-t)$$ and the squared distance energy  of $K_{n_1,n_2,\cdots,n_t}$ is given by $$E_{\Delta}(G)= 8(n-t).$$
\end{theorem}
\begin{proof}
Let $G=K_{n_1,n_2,\cdots,n_t}$. Without loss of generality, let us assume $n_1 \ge n_2 \ge \cdots \ge n_t \geq 2$. Using Eqn.~\eqref{eqn:D(K)}, we have
\begin{equation}\label{eqn:rank-pert1}
\Delta(G) + 4 I_n = Diag \left( 3J_{n_1},3J_{n_2}, \dots , 3J_{n_t} \right)+J_{n}.
\end{equation}
For  $1\leq i \leq t$, the eigenvalues of $J_{n_i}$ are $n_i$ with  multiplicity $1$ and $0$ with  multiplicity $n_i-1.$ Therefore, the eigenvalues of $Diag \left( 3J_{n_1},3J_{n_2}, \dots , 3J_{n_t} \right)$ are $3n_1,3n_2,\ldots, 3n_t$ each with multiplicity $1$ and $0$ with  multiplicity $n-t$. 

 Let  $\lambda_1\geq \lambda_2 \geq \cdots \geq \lambda_n $ be the eigenvalues of  $\Delta(G) + 4 I_n$. From Eqn.~\eqref{eqn:rank-pert1}, $\Delta(G) + 4 I_n$ is a rank one  perturbation of $Diag \left( J_{n_1},J_{n_2}, \dots , J_{n_t} \right)$. Then, by Lemma~\ref{lem:rank-pert-ev}, we have
$$\lambda_1 \geq 3n_1 \geq \lambda_2 \geq 3n_2 \geq \cdots \geq \lambda_t \geq 3n_t \geq  \lambda_{t+1}\geq 0 \geq  \lambda_{t+2}\geq  0 \geq \cdots \geq  \lambda_n \geq 0. $$
Since $n_i\geq 2$ for $1\leq i \leq t$, so  $\lambda_1, \lambda_2,\cdots, \lambda_t \geq 6 $. Moreover,  Eqn.~\eqref{eqn:rank-pert1} yields that the rank of  $\Delta(G)+4 I_n$ is $t$ and hence $\lambda_i =0$ for $t+1\leq i \leq n$. Therefore, the eigenvalues of  $\Delta(G)$ are $\lambda_1 -4,\lambda_2 -4,\cdots, \lambda_t -4$ each with  multiplicity $1$ and $-4$ with multiplicity $n-t$.  This implies that   $\lambda_1 -4,\lambda_2 -4,\cdots, \lambda_t -4$ are all  positive  eigenvalues of $\Delta(G)$, and  the only negative eigenvalue of $\Delta(G)$ is $-4$ with multiplicity $n-t$. Hence,  $\textup{In}(\Delta(G))=(t, 0, n-t)$  and the squared distance energy  of $G$  is given by $$E_{\Delta}(G)=2(4(n-t)) =8(n-t).$$
This completes the proof.
\end{proof}
\begin{cor}\label{cor:monic-char}
Let $t\geq 2$ and $n_1,n_2,\cdots,n_t$ be positive integers  such that  $n_i\geq 2$ for $1\leq i\leq t$. If $B_t$ is the matrix  defined in Eqn~(\ref{eqn:M_2}) with parameters $n_1,n_2,\cdots,n_t$, then  the eigenvalues of $B_t$ are positive.
\end{cor}
\begin{proof}
 Let $G=K_{n_1,n_2,\cdots,n_t}$ and $n=\sum_{i=1}^t n_i$. If $\lambda_1(G)\geq \lambda_2(G)\geq \cdots \geq \lambda_n(G) $ are the eigenvalues of $\Delta(G)$, then by Theorem~\ref{thm:ni-geq-2}, 
  $-4$ is an eigenvalue of $\Delta(G)$ with multiplicity $n-t$ and  $\lambda_i(G)>0$ for $1\leq i\leq t$. Therefore, in view of Proposition~\ref{prop:char_eqn} and Remark~\ref{rem:monic-char1}, we get $\lambda_i(G)$ for $1\leq i\leq t$ are the roots of $\det(xI_t - B_t)= 0$ and hence the result follows.
\end{proof}

Now we prove a few lemmas useful to find the inertia of $\Delta(K_{n_1,n_2,\cdots,n_t})$ and   $E_{\Delta}(K_{n_1,n_2,\cdots,n_t})$, whenever $K_{n_1,n_2,\cdots,n_t}$ belongs to  the class  $\mathcal{M}(n,t,h)$.

\begin{lem}\label{lem:f-neq-0}
 Let $n_1,n_2,\cdots,n_s$ be positive integers  such that   $n_i\geq 2$ for $1\leq i\leq s$ and $h$ be a positive integer. Let $B_s$ be the matrix  defined in Eqn~(\ref{eqn:M_2}) with parameters $n_1,n_2,\cdots,n_s$.  If  $f(x)= (x+1) \det(xI_s-B_s)- h  \prod_{i=1}^s (x+4-3n_i)$, then $f(x)\neq 0$ for   $x\leq -1.$ 
\end{lem}
\begin{proof}
For   $x\leq -1$ and  $n_i\geq 2$ for $1\leq i\leq s$, we have $x+4-3n_i <0$  and hence the sign of $-h  \prod_{i=1}^s (x+4-3n_i)$ is $(-1)^{s+1}$.  By  Corollary~\ref{cor:monic-char}, the eigenvalues  of $B_s$ are positive which implies that  the sign of $(x+1) \det(xI_s-B_s)$ is $(-1)^{s+1}$ for $x< -1$ and the value of $(x+1) \det(xI_s-B_s)$ is zero for $x=1$. Therefore,  both  $(x+1) \det(xI_s-B_s)$ and $-h  \prod_{i=1}^s (x+4-3n_i)$ are of the  same sign and hence $f(x)\neq 0$ for   $x\leq -1.$
\end{proof}

\begin{lem}\label{lem:eigen-sq-dist}
Let $t=s+h\geq 2$ and   $n_1,n_2,\cdots,n_t$ be positive integers  such that  $n_i\geq 2$ for $1\leq i\leq s$ and $n_i=1$ for $s+1 \leq i\leq t=s+h$. 
Let $G$ be the complete $t$-partite graph $K_{n_1,n_2,\cdots,n_t}$ on $n=\sum_{i=1}^t n_i$ vertices and $\Delta(G)$ be the squared distance matrix of $G$. If $\lambda_1(G)\geq \lambda_2(G)\geq \cdots \geq \lambda_n(G) $ are the eigenvalues of $\Delta(G)$, then 
\begin{enumerate}
\item[$(i)$]  $\lambda_i(G)>0$ for $1\leq i\leq s.$
\item[$(ii)$]   $\lambda_{s+1}(G)>0  \textup{ if and only if  } h-1 > \ds\sum_{i=1}^s \dfrac{n_i}{3n_i-4},$ $\lambda_{s+1}(G)=0  \textup{ if and only if  } h-1 = \ds\sum_{i=1}^s \dfrac{n_i}{3n_i-4}$  and  $\lambda_{s+1}(G)<0  \textup{ if and only if  } h-1 < \ds\sum_{i=1}^s \dfrac{n_i}{3n_i-4}.$ 

Moreover, if $\lambda_{s+1}(G) <0,$ then $ -1<\lambda_{s+1}(G) <0.$
\item[$(iii)$]   $\lambda_i(G)=-1$ for $s+2\leq i\leq s+h$ and $\lambda_i(G)=-4$ for $s+h+1\leq i\leq n.$
\end{enumerate}
\end{lem}
\begin{proof}
Let $H=K_{n_1,n_2,\cdots,n_s}$ be the complete $s$-partite graph on $n-h=\sum_{i=1}^s n_i$ vertices and  $\Delta(H)$ be the squared distance matrix of $H$. Then,  $\Delta(H)$ is a principal submatrix of $\Delta(G)$. If $\lambda_1(H)\geq \lambda_2(H)\geq \cdots \geq \lambda_{n-h}(H) $ denote the eigenvalues of    $\Delta(H)$, then  Theorem~\ref{thm:ni-geq-2} yields that $\lambda_s(H)>0$ and hence using Theorem~\ref{thm:inter}, we have  $\lambda_s(G)\geq \lambda_s(H)>0$. This proves part $(i)$.

Let $B_s$ be the matrix defined in Eqn~(\ref{eqn:M_2}) with parameters $n_1\geq n_2\geq \cdots \geq  n_s \geq 2$.  By Corollary~\ref{cor:char-reduce}, the characteristic polynomial of $\Delta(G)$ is given by  $$P_{\Delta}(G,x)=(x+4)^{n-t}(x+1)^{h-1} p(G,x),$$ where
$ p(G,x)= (x+1) \det(xI-B_s)- h  \prod_{i=1}^s (x+4-3n_i).$  By Lemma~\ref{lem:f-neq-0}, we have  $p(G,x)\neq 0$ for  $x\leq -1$, and hence  the roots of the polynomial $p(G,x)$ are $\lambda_i(G)$ for $1\leq i\leq s+1$.
Therefore,  $\lambda_i(G)=-1$ for $s+2\leq i\leq s+h$ and $\lambda_i(G)=-4$ for $s+h+1\leq i\leq n.$ Since  by part $(i)$, we know that $\lambda_i(G)>0$ for $1\leq i\leq s$, so  if $\lambda_{s+1}(G) <0,$ then $ -1<\lambda_{s+1}(G) <0.$ 

Now substituting $n_i=1$ for $s+1 \leq i\leq t=s+h$ in  Corollary~\ref{cor:detD_single}, the determinant of the squared distance matrix $\Delta(G)$ is given by
\begin{align}\label{eqn:det-h-1}
\det \Delta(G) &= (-4)^{n-t}\left[ \sum_{i=1}^t \Bigg( n_i\prod_{j \neq i}(3n_j - 4) \Bigg) +  \prod_{i=1}^t (3n_i - 4)\right]\nonumber\\
               &= (-4)^{n-t}\left[ (-1)^h\sum_{i=1}^s \Bigg( n_i\prod_{j \neq i}(3n_j - 4) \Bigg) + (-1)^{h-1} h \prod_{i=1}^s (3n_i - 4)+ (-1)^h \prod_{i=1}^s (3n_i - 4)\right]\nonumber\\
               &=(-4)^{n-t} (-1)^{h-1} \left[ (h-1)\prod_{i=1}^s (3n_i - 4) - \sum_{i=1}^s \Bigg( n_i\prod_{j \neq i}(3n_j - 4) \Bigg)\right]\nonumber\\
               &=(-4)^{n-t} (-1)^{h-1} \left[ (h-1)- \sum_{i=1}^s \frac{n_i}{3n_i - 4}\right] \prod_{i=1}^s (3n_i - 4).
\end{align}

Further, we also know that $\lambda_i(G)=-1$ for $s+2\leq i\leq s+h$, {\it i.e.,} $-1$ is an eigenvalue  with multiplicity $h-1$ and $\lambda_i(G)=-4$ for $s+h+1\leq i\leq n$, {\it i.e.,}  $-4$ is an eigenvalue  with multiplicity $n-t$. Thus, using $ \det \Delta(G) = \prod_{i=1}^n \lambda_i(G)$, the Eqn.~\eqref{eqn:det-h-1} yields that
$$\prod_{i=1}^{s+1} \lambda_i(G) =\left[ (h-1)- \sum_{i=1}^s \frac{n_i}{3n_i - 4}\right] \prod_{i=1}^s (3n_i - 4).$$
Since $\prod_{i=1}^s (3n_i - 4) >0$ as $n_i\geq 2$ for  $1\leq i\leq s$ and by part $(i)$, $\lambda_i(G)>0$ for $1\leq i\leq s$, so  $\lambda_{s+1}(G)>0  \textup{ if and only if  } h-1 > \ds\sum_{i=1}^s \dfrac{n_i}{3n_i-4},$ $\lambda_{s+1}(G)=0  \textup{ if and only if  } h-1 = \ds\sum_{i=1}^s \dfrac{n_i}{3n_i-4}$  and  $\lambda_{s+1}(G)<0  \textup{ if and only if  } h-1 < \ds\sum_{i=1}^s \dfrac{n_i}{3n_i-4}.$ This completes the proof.
\end{proof}

The following result determines the inertia and the squared distance energy whenever complete $t$-partite graphs are in   $\mathcal{M}(n,t,h)$. This result is an immediate consequence of Lemma~\ref{lem:eigen-sq-dist}, and hence we state the result without proof. 

\begin{theorem}\label{thm:inertia-energy}
Let $t=s+h\geq 2$ and   $n_1,n_2,\cdots,n_t$ be positive integers  such that $n_i\geq 2$ for $1\leq i\leq s$  and $n_i=1$ for $s+1 \leq i\leq t=s+h$. 
Let $G$ be the complete $t$-partite graph $K_{n_1,n_2,\cdots,n_t}$ on $n=\sum_{i=1}^t n_i$ vertices. Let $\Delta(G)$ be the squared distance matrix of $G$ and $\lambda_1(G)\geq \lambda_2(G)\geq \cdots \geq \lambda_n(G) $ are the eigenvalues of $\Delta(G)$.  Then, the inertia of $\Delta(G)$ is given by
$$\textup{In}(\Delta(G))= 
\begin{cases}
(s+1, 0, n-s-1) & \textup{ if } h-1 > \ds\sum_{i=1}^s \dfrac{n_i}{3n_i-4},\\
(s, 1, n-s-1) & \textup{ if } h-1 = \ds \sum_{i=1}^s \dfrac{n_i}{3n_i-4},\\
(s, 0, n-s) & \textup{ if } h-1 < \ds \sum_{i=1}^s \dfrac{n_i}{3n_i-4},\\
\end{cases}$$ 
and  the squared distance energy  of $G$ is given by $$E_{\Delta}(G)= \begin{cases}
8(n-t)+2(h-1) & \textup{ if } h-1 \geq  \ds\sum_{i=1}^s \dfrac{n_i}{3n_i-4},\\
8(n-t)+2[(h-1)+ \theta] & \textup{ if } h-1 < \ds \sum_{i=1}^s \dfrac{n_i}{3n_i-4},\\
\end{cases}$$
where  $0< \theta < 1$ and  $\theta= -\lambda_{s+1}(G)$. Equivalently, the squared distance energy  of $G$ is given by
$$E_{\Delta}(G)=\begin{cases}
8(n-t)+2(h-1) & \textup{ if }\lambda_{s+1}(G) \geq 0,\\
8(n-t)+2[(h-1)-\lambda_{s+1}(G) ] & \textup{ if } \lambda_{s+1}(G)<0.\\
\end{cases}$$ 
\end{theorem}

\begin{cor}\label{cor:energy-bound}
Let $K_{n_1,n_2,\cdots,n_t}$  be a complete $t$-partite graphs on $n=\sum_{i=1}^t n_i$ vertices. If $h= |\{i : n_i=1\}|$ and $h\geq 1$, then 
$$ 8(n-t)+2(h-1) \leq E_{\Delta}(K_{n_1,n_2,\cdots,n_t}) < 8(n-t)+2h.$$
\end{cor}

\begin{cor}\label{cor:big-h-energy}
Let  $n=\sum_{i=1}^t n_i =\sum_{i=1}^t m_i$. Let $G=K_{n_1,n_2,\cdots,n_t}$ and $H=K_{m_1,m_2,\cdots,m_t}$ be complete $t$-partite graphs on $n$ vertices. Let $h_G= |\{i : n_i=1\}|$ and $h_H= |\{i : m_i=1\}|$. If $ h_H < h_G$, then $   E_{\Delta}(H)< E_{\Delta}(G)$.
\end{cor}
\begin{proof}
For $h_G\geq 2$, using Theorem~\ref{thm:inertia-energy}, it is easy to see that the result hold true. Next, let $h_G=1$. Then $h_H=0$ and hence $m_i\geq 2$ for $1\leq i \leq t$. By Theorem~\ref{thm:ni-geq-2}, we have $E_{\Delta}(H)=8(n-t)$.

Without loss of generality, let us assume $n_1\geq n_2\geq \cdots \geq n_t$. Thus, using $h_G=1$, we get $n_i\geq 2$ for $1\leq i \leq t-1$ and therefore, $0=h_G-1<\sum_{i=1}^{t-1} \dfrac{n_i}{3n_i-4} $. By Theorem~\ref{thm:inertia-energy}, we have $E_{\Delta}(G)=8(n-t)+ \theta$, where $0<\theta<1$. Hence $  E_{\Delta}(H)< E_{\Delta}(G)$.
\end{proof}

\section{Extremal Graphs for Maxima and Minima of $E_{\Delta}(K_{n_1,n_2,\cdots,n_t})$ }\label{sec:max-min-energy}

In this section, we prove that if $ (m_1,m_2,\cdots,m_t)\succ (n_1,n_2,\cdots,n_t)$, then $E_{\Delta}(K_{m_1,m_2,\cdots,m_t}) \geq E_{\Delta}(K_{n_1,n_2,\cdots,n_t})$. Furthermore,   for fixed value of $n$ and $t$, we find  complete $t$-partite graphs for which the squared distance energy $E_{\Delta}(K_{n_1,n_2,\cdots,n_t})$ attain its maxima and minima, and discuss the uniqueness of these extremal graphs. We first prove a few  lemmas  useful to achieve the above goals.

\begin{lem}\label{lem:ineq-en}
Let $n_p , n_q \geq 2$  and  $n_p \geq n_q+2$.  If $0\leq x<1$, then
\begin{equation}\label{eqn_energy_ineq2}
\frac{n_p}{x-4+3n_p} + \frac{n_q}{x-4+3n_q} > \frac{n_p-1}{x-4+3(n_p-1)} + \frac{n_q+1}{x-4+3(n_q+1)}.
\end{equation}
\end{lem}
\begin{proof}
Observe that, if $n_p , n_q \geq 2$  and  $n_p \geq n_q+2$, then  $(x-4+3n_p), (x-4+3n_q),  (x-4+3(n_p-1)), (x-4+3(n_q+1)) > 0$ for $0\leq x<1$. For  a fix $x \in [0,1)$,  let  $a= x-4+3n_p$ and $ b=x-4+3n_q$. Then, in view of part  $(ii)$ of Lemma~\ref{lem:ab-ineq}, to prove inequality  Eqn.~\eqref{eqn_energy_ineq2} it is enough to show that  
$$3[n_q a(a-3)-n_p b(b+3)] - ab(a-b-6)>0.$$
Note that,

 \vspace*{-0.25cm}
\begin{align*}
 & \quad \ 3[n_q a(a-3)-n_p b(b+3)]\\
&=3\left[n_q(x-4+3n_p)(x-4+3(n_p-1))- n_p(x-4+3n_q)(x-4+3(n_q+1)) \right]\\
&= 3\left[ n_q (x-4)^2 +3(x-4)(2n_pn_q-n_q)+9 n_p(n_p-1) \right]\\
&\hspace*{3cm}-3\left[n_p(x-4)^2 +3(x-4)(2n_pn_q + n_p)+9 n_pn_q(n_q+1) \right]\\
&=3\left[ (n_q-n_p)(x-4)^2-3(x-4)(n_p+n_q) + 9n_pn_q(n_p-n_q-2)\right],
\end{align*}
and
\begin{align*}
ab(a-b-6)=& 3(x-4+3n_p) (x-4+3n_q)(n_p-n_q-2)\\
=& 3 \left[(n_p-n_q-2) (x-4)^2  + 3(x-4)(n_p+n_q)(n_p-n_q-2) + 9n_pn_q(n_p-n_q-2) \right].
\end{align*}
Thus,
\begin{align}\label{eqn:ab-1}
 & \quad \ 3[n_q a(a-3)-n_p b(b+3)] - ab(a-b-6)\nonumber\\
 &= 3\left[ (n_q-n_p)(x-4)^2-3(x-4)(n_p+n_q)\right]\nonumber\\
 & \hspace*{3cm} -  3 \left[(n_p-n_q-2) (x-4)^2  + 3(x-4)(n_p+n_q)(n_p-n_q-2) \right] \nonumber\\
 &= 3\left[  2 (x-4)^2(n_q-n_p+1) + 3(x-4)(n_p+n_q)(n_q-n_p+1) \right]\nonumber\\
 &= (x-4)(n_q-n_p+1) [2(x-4)+3(n_p+n_q)].
\end{align}
Further, for $0\leq x <1$, we have $ -4 \leq x-4 <-3$, and using $n_p\geq n_q +2$, we have $n_q-n_p+1 \leq -1$, which implies that 
\begin{equation}\label{eqn:ab-2}
(x-4)(n_q-n_p+1) >0.
\end{equation}
Next, for $0\leq x <1$, we have $ -8 \leq 2(x-4) <-6$ and 
 for $n_p,n_q \geq 2$, we have $3(n_p+n_q) \geq 12$, which implies that 
 \begin{equation}\label{eqn:ab-3}
 [2(x-4)+3(n_p+n_q)]>0.
 \end{equation}
Using Eqns.~\eqref{eqn:ab-1} -~\eqref{eqn:ab-3}, we have $3[n_q a(a-3)-n_p b(b+3)] - ab(a-b-6)>0.$ This completes the proof. 
\end{proof}

\begin{assume}\label{asmp}
Let $s\geq 2$ and $h\geq 1$. Let $t=s+h\geq 3$ and   $n_1,n_2,\cdots,n_t$ be positive integers  such that  $n_i  \geq 2$ for $1 \leq i\leq s$ and $n_i=1$ for $s+1 \leq i\leq t=s+h$.  Let  $n=\sum_{i=1}^t n_i$ and  $n_p \ge n_q+2$ for $1\leq p ,q\leq s$. Let $G=K_{n_1,\cdots,n_p,\cdots,n_q,\cdots,n_s,n_{s+1},\cdots, n_t}$ and $H=K_{n_1,\cdots,n_p-1,\cdots,n_q+1,\cdots,n_s,n_{s+1},\cdots, n_t}$  be complete $t$-partite graphs on $n$ vertices. Let $\lambda_1(G)\geq \lambda_2(G)\geq \cdots \geq \lambda_n(G) $  and $\lambda_1(H)\geq \lambda_2(H)\geq \cdots \geq \lambda_n(H) $ be the eigenvalues of the squared distance matrix  $ \Delta(G)$ and $\Delta(H)$, respectively.
\end{assume}

\begin{lem}\label{lem:pert-non-negative}
Under the Assumption~\ref{asmp}, if $ \lambda_{s+1}(G)\geq 0$, then $ \lambda_{s+1}(H)>0$ and 
$$E_{\Delta}(G)= E_{\Delta}(H)= 8(n-t)+2(h-1).$$
\end{lem}
\begin{proof}
By part ($ii$) of Lemma~\ref{lem:eigen-sq-dist} we know that  $\lambda_{s+1}(G)\geq 0$ only if  $h-1 \geq  \ds\sum_{i=1}^s \dfrac{n_i}{3n_i-4}.$ For $x=0$,  Lemma~\ref{lem:ineq-en} gives us
$$\frac{n_p}{3n_p-4} + \frac{n_q}{3n_q-4} > \frac{n_p-1}{3(n_p-1)-4} + \frac{n_q+1}{3(n_q+1)-4},$$
and hence 
$$ h-1 \geq  \sum_{i=1}^s \dfrac{n_i}{3n_i-4} > 
\frac{n_p-1}{3(n_p-1)-4} + \frac{n_q+1}{3(n_q+1)-4}+ \sum_{i=1 \atop i\neq p,q}^s \dfrac{n_i}{3n_i-4}.
$$
Therefore, using part ($ii$) of Lemma~\ref{lem:eigen-sq-dist}, we get $ \lambda_{s+1}(H)>0$. Furthermore, by Theorem~\ref{thm:inertia-energy}, we have $E_{\Delta}(G)= E_{\Delta}(H)= 8(n-t)+2(h-1).$ 
\end{proof}

\begin{lem}\label{lem:pert-negative}
Under the Assumption~\ref{asmp}, if $ \lambda_{s+1}(G)<0$, then $\lambda_{s+1}(G) < \lambda_{s+1}(H)$ and   $$E_{\Delta}(G)> E_{\Delta}(H).$$
\end{lem}
\begin{proof}
Since  $ \lambda_{s+1}(G)<0$, so by part ($ii$) of Lemma~\ref{lem:eigen-sq-dist} we know that $ -1<\lambda_{s+1}(G)<0$ and by Theorem~\ref{thm:inertia-energy}, we have $$E_{\Delta}(G)= 8(n-t)+2[(h-1) - \lambda_{s+1}(G)].$$
If  $\lambda_{s+1}(H)\geq 0$, then  by Theorem~\ref{thm:inertia-energy}, we have $E_{\Delta}(H)= 8(n-t)+2(h-1)$ and the result holds true. Next,  if $\lambda_{s+1}(H)< 0$, then by part ($ii$) of Lemma~\ref{lem:eigen-sq-dist} we know that $ -1<\lambda_{s+1}(H)<0$ and by Theorem~\ref{thm:inertia-energy}, we get $E_{\Delta}(H)= 8(n-t)+2[(h-1) - \lambda_{s+1}(H)].$ Therefore, in this case,  to complete the proof we  show that $ -1<\lambda_{s+1}(G)< \lambda_{s+1}(H)<0$.

For$ 1\leq i\leq  s$, let  $p(G,x)$ and  $p(H,x)$ be the monic polynomials  defined in Eqn.~\eqref{eqn:char-reduce-monic1} with parameters $n_1,\cdots,n_p,\cdots,n_q,\cdots,n_s$ and $n_1,\cdots,n_p-1,\cdots,n_q+1,\cdots,n_s$, respectively. Then, by Corollary~\ref{cor:char-reduce}, the characteristic polynomials of $\Delta(G)$ and $\Delta(H)$ are given by
$$P_{\Delta}(G,x)=(x+4)^{n-t}(x+1)^{h-1} p(G,x) \mbox{ and } P_{\Delta}(H,x)=(x+4)^{n-t}(x+1)^{h-1} p(H,x),$$
respectively. By our hypothesis $\lambda_{s+1}(G),\lambda_{s+1}(H)\in (-1,0)$ and hence  by Lemma~\ref{lem:eigen-sq-dist}  we know that $\lambda_{s+1}(G)$ and $\lambda_{s+1}(H)$ are  the unique negative roots of $p(G,x)$ and  $p(H,x)$, respectively. Let us define 
\begin{equation}\label{eqn:f-g}
f(x)= (-1)^{s+1}p(G,-x) \mbox{ and } g(x)=(-1)^{s+1}p(H,-x).
\end{equation}
By construction, $f(x)$ and $g(x)$ are monic polynomials. Furthermore, $-\lambda_{s+1}(G)$ and $-\lambda_{s+1}(H)$ are the largest roots of $f(x)$ and $g(x)$, respectively  and both of them lie in the interval $(0,1)$.  Therefore, in view of Lemma~\ref{lem:comp-r-root},  to prove  $ -1<\lambda_{s+1}(G)< \lambda_{s+1}(H)<0$ or equivalently   $ 0< - \lambda_{s+1}(H)<- \lambda_{s+1}(G)<1$, we show that $f(-\lambda_{s+1}(H))<0$ and we proceed as follows.\\

\noindent{\bf Claim:} $f(-\lambda_{s+1}(H))<0.$\\

Observe that, using $n_i\geq 2$ for $ 1\leq i\leq  s$ and $n_p\geq n_q +2$, we have  $x+4-3n_i<0$  for $ 1\leq i\leq  s$, $x+4-3(n_p-1)<0$ and $x+4-3(n_q+1)<0$ for $-1< x< 0$ ({\it i.e.},  $x+4-3(n_p-1)$,  $x+4-3(n_q+1)$ and $x+4-3n_i$  for $ 1\leq i\leq  s$ are non-zero, whenever $-1< x< 0$). Therefore,  for $ x\in (-1,0)$, the monic polynomials $p(G,x)$ and $p(H,x)$   as defined in Eqn.~\eqref{eqn:char-reduce-monic1} can be written as
{\small 
\begin{align*}
p(G,x)&=(x+4-3n_p)(x+4-3n_q) \prod_{i=1 \atop i\neq p,q}^s (x+4-3n_i) \\
& \ \  \times \Bigg[(x+1-h)-(x+1)\Bigg(\frac{n_p}{x+4-3n_p}+  \frac{n_q}{x+4-3n_q}\Bigg)-(x+1)\sum_{i=1 \atop i\neq p,q}^s \frac{n_i}{x+4-3n_i}\Bigg]  \mbox{ and }\\
p(H,x)&=(x+4-3(n_p-1))(x+4-3(n_q+1)) \prod_{i=1 \atop i\neq p,q}^s (x+4-3n_i) \\
& \ \  \times \Bigg[(x+1-h)-(x+1)\Bigg(\frac{n_p-1}{x+4-3(n_p-1)}+  \frac{n_q+1}{x+4-3(n_q+1)}\Bigg)-(x+1)\sum_{i=1 \atop i\neq p,q}^s \frac{n_i}{x+4-3n_i}\Bigg].
\end{align*}}
Then, for $x\in (0,1)$, the monic polynomials $f(x)$ and $g(x)$  defined in   Eqn.~\eqref{eqn:f-g}, can be written as 
{\small 
\begin{align*}
f(x)&=(x-4+3n_p)(x-4+3n_q) \prod_{i=1 \atop i\neq p,q}^s (x-4+3n_i) \\
& \ \  \times \Bigg[(x-1+h)+(x-1)\Bigg(\frac{n_p}{x-4+3n_p}+  \frac{n_q}{x-4+3n_q}\Bigg)+(x-1)\sum_{i=1 \atop i\neq p,q}^s \frac{n_i}{x-4+3n_i}\Bigg]  \mbox{ and }\\
g(x)&=(x-4+3(n_p-1))(x-4+3(n_q+1)) \prod_{i=1 \atop i\neq p,q}^s (x-4+3n_i) \\
& \ \  \times \Bigg[(x-1+h)+(x-1)\Bigg(\frac{n_p-1}{x-4+3(n_p-1)}+  \frac{n_q+1}{x-4+3(n_q+1)}\Bigg)+(x-1)\sum_{i=1 \atop i\neq p,q}^s \frac{n_i}{x-4+3n_i}\Bigg].
\end{align*}}
Let us denote
{\small
\begin{align*}
\varphi_f(x)&=(x-1+h)+(x-1)\Bigg(\frac{n_p}{x-4+3n_p}+  \frac{n_q}{x-4+3n_q}\Bigg)+(x-1)\sum_{i=1 \atop i\neq p,q}^s \frac{n_i}{x-4+3n_i}, \\
\varphi_g(x)&= (x-1+h)+(x-1)\Bigg(\frac{n_p-1}{x-4+3(n_p-1)}+  \frac{n_q+1}{x-4+3(n_q+1)}\Bigg)+(x-1)\sum_{i=1 \atop i\neq p,q}^s \frac{n_i}{x-4+3n_i}.
\end{align*}}
Thus, for $x\in (0,1)$, we have
\begin{equation}\label{eqn:f-g-phi}
\begin{cases}
f(x)= \ds \Bigg[(x-4+3n_p)(x-4+3n_q) \prod_{i=1 \atop i\neq p,q}^s (x-4+3n_i)\Bigg] \varphi_f(x),\\
g(x)=\ds \Bigg[ (x-4+3(n_p-1))(x-4+3(n_q+1)) \prod_{i=1 \atop i\neq p,q}^s (x-4+3n_i) \Bigg] \varphi_g(x).
\end{cases}
\end{equation}
Using  $n_i\geq 2$ for $ 1\leq i\leq  s$ and $n_p\geq n_q +2$, we have  
\begin{equation}\label{eqn:1}
x-4+3(n_p-1) >0, x-4+3(n_q+1)>0, x-4+3n_i>0 \mbox{ for } 1\leq i \leq s \mbox{ and } x\in (0,1).
\end{equation}
Therefore, in view of   Eqns.~\eqref{eqn:f-g-phi} and~\eqref{eqn:1}, we get $f(x)=0$ if and only if  $\varphi_f(x)=0$ and $g(x)=0$ if and only if  $\varphi_g(x)=0$ for $x\in (0,1)$. Further,  for $x\in (0,1)$, we have $(x-1)<0$ and  Lemma~\ref{lem:ineq-en} yields that
\begin{equation*}
(x-1)\Bigg(\frac{n_p}{x-4+3n_p}+  \frac{n_q}{x-4+3n_q}\Bigg) <(x-1)\Bigg(\frac{n_p-1}{x-4+3(n_p-1)}+  \frac{n_q+1}{x-4+3(n_q+1)}\Bigg),
\end{equation*}
and hence for all $x\in (0,1)$, we have

\begin{align*}\label{eqn:3}
  &\Bigg[(x-1+h)+(x-1)\Bigg(\frac{n_p}{x-4+3n_p}+  \frac{n_q}{x-4+3n_q}\Bigg)+(x-1)\sum_{i=1 \atop i\neq p,q}^s \frac{n_i}{x-4+3n_i}\Bigg] \nonumber\\
< & \Bigg[(x-1+h)+(x-1)\Bigg(\frac{n_p-1}{x-4+3(n_p-1)}+  \frac{n_q+1}{x-4+3(n_q+1)}\Bigg)+(x-1)\sum_{i=1 \atop i\neq p,q}^s \frac{n_i}{x-4+3n_i} \Bigg]. 
\end{align*}
That is,  $\varphi_f(x)< \varphi_g(x)$ for all $x\in (0,1)$. In particular, $0<-\lambda_{s+1}(H)<1$ and $g(-\lambda_{s+1}(H))=0,$ which implies that  $\varphi_f(-\lambda_{s+1}(H))< \varphi_g(-\lambda_{s+1}(H))=0$. Furthermore, by Eqn.~\eqref{eqn:1} we know that $x-4+3n_i>0 \mbox{ for } 1\leq i \leq s \mbox{ and } x\in (0,1),$  and hence using Eqn.~\eqref{eqn:f-g-phi}, we get $f(-\lambda_{s+1}(H))< 0$.  This completes the proof.
\end{proof}

Now we state the following result (without proof) which is obtained by combining the conclusions of  Lemmas~\ref{lem:pert-non-negative} and~\ref{lem:pert-negative}.

\begin{lem}\label{lem:pert-energy-h}
Let $s \geq 2$ and $h\geq 1$. Let $n_1,n_2,\cdots,n_s$ be positive integers  such that  $n_i  \geq 2$ for $1 \leq i\leq s$.  If   $n_p \ge n_q+2$ for $1\leq p ,q\leq s$,  then 
 $$E_{\Delta}(K_{n_1,\cdots,n_p,\cdots,n_q,\cdots,n_s,\scriptsize{\underbrace{1,1,\ldots,1}_{h\ times}}}) \geq E_{\Delta}(K_{n_1,\cdots,n_p-1,\cdots,n_q+1,\cdots,n_s,\scriptsize{\underbrace{1,1,\ldots,1}_{h\ times}}}).$$
\end{lem}

\begin{theorem}\label{thm:energy-mojor}
Let $t=s+h$ with $s \geq 2$ and $h\geq 1$. Let  $n_1,n_2,\cdots,n_t$ and  $m_1,m_2,\cdots,m_t$  be positive integers. Then the following results hold.
\begin{enumerate}
\item[$(i)$]   Let $n_i  \geq 2$ and $m_i  \geq 2$ for $1 \leq i\leq s$. If  $ (m_1,m_2,\cdots,m_s,\scriptsize{\underbrace{1,1,\ldots,1}_{h\ times}})\succ (n_1,n_2,\cdots,n_s,\scriptsize{\underbrace{1,1,\ldots,1}_{h\ times}})$, then $$E_{\Delta}(K_{m_1,m_2,\cdots,m_s,\scriptsize{\underbrace{1,1,\ldots,1}_{h\ times}}}) \geq E_{\Delta}(K_{n_1,n_2,\cdots,n_s,\scriptsize{\underbrace{1,1,\ldots,1}_{h\ times}}}).$$

\item[$(ii)$] If  $ (m_1,m_2,\cdots,m_t)\succ (n_1,n_2,\cdots,n_t)$, then $E_{\Delta}(K_{m_1,m_2,\cdots,m_t}) \geq E_{\Delta}(K_{n_1,n_2,\cdots,n_t}).$
\end{enumerate}
\end{theorem}
\begin{proof}
Let $t=s+h$ and $s<t$.  Without loss of generality, let us assume $n_1\geq n_2\geq \cdots \geq n_s\geq 2$ and  $m_1\geq m_2\geq \cdots \geq m_s \geq 2$. Let  $ (m_1,m_2,\cdots,m_s,\scriptsize{\underbrace{1,1,\ldots,1}_{h\ times}})\succ (n_1,n_2,\cdots,n_s,\scriptsize{\underbrace{1,1,\ldots,1}_{h\ times}})$. Thus, if we denote $Y=(m_1,m_2,\cdots,m_s)$ and $X=(n_1,n_2,\cdots,n_s) $, then by  Theorem~\ref{thm:maj_tuple}, we have $ X \succ Y.$ 

Recall that, $\mathcal{B}_n = \{(z_1,\cdots,z_n) \in \mathbb{Z}^n : z_1 \ge z_2 \ge \cdots \ge z_n\}$. For $Z=(z_1,\cdots,z_s) \in \mathcal{B}_s$, we denote $\widetilde{Z}=(z_1,\cdots,z_s, \scriptsize{ \underbrace{1,1,\ldots,1}_{h\ times}}) \in \mathcal{B}_{s+h}$. Since $X, Y \in \mathcal{B}_s$ and $ X \succ Y$, so by Theorem~\ref{thm:maj_seq}, there exists a sequence in $\mathcal{B}_s$ such as $Y = Y_0 \succ Y_1 \succ \cdots \succ Y_l = X$ such that $ Y_i = Y_{i-1} + E_{j_i,k_i}^s$ for some $k_i > j_i, \; 1\leq i \leq l$. Therefore, using Theorem~\ref{thm:maj_tuple}, we have
\begin{equation}\label{eqn:mojor-tilde}
(m_1,m_2,\cdots,m_s,\underbrace{1,1,\ldots,1}_{h\ times})=\widetilde{Y} = \widetilde{Y}_0 \succ \widetilde{Y}_1 \succ \cdots \succ \widetilde{Y}_l = \widetilde{X}=(n_1,n_2,\cdots,n_s,\underbrace{1,1,\ldots,1}_{h\ times}),
\end{equation}
where  $ \widetilde{Y}_i = \widetilde{Y}_{i-1} + E_{j_i,k_i}^{s+h}$ for some $k_i > j_i; 1\leq i \leq l$. Therefore,  applying Lemma~\ref{lem:pert-energy-h}  inductively on Eqn.~\eqref{eqn:mojor-tilde} yields that $E_{\Delta}(K_{m_1,m_2,\cdots,m_s,\scriptsize{\underbrace{1,1,\ldots,1}_{h\ times}}}) \geq E_{\Delta}(K_{n_1,n_2,\cdots,n_s,\scriptsize{\underbrace{1,1,\ldots,1}_{h\ times}}})$. This proves part $(i)$.

Next, let $ (m_1,m_2,\cdots,m_t)\succ (n_1,n_2,\cdots,n_t)$. Let $G=K_{m_1,m_2,\cdots,m_t}$ and $H=K_{n_1,n_2,\cdots,n_t}.$ Let us denote $h_G=|\{i: m_i=1\}|$ and $h_H=|\{i: n_i=1\}|$. Then $ (m_1,m_2,\cdots,m_t)\succ (n_1,n_2,\cdots,n_t)$ implies that   either $h_G=h_H$ or $h_G>h_H$. Thus, if  $h_G=h_H$, then by part $(i)$, we have $E_{\Delta}(G) \geq E_{\Delta}(H)$ and if $h_G>h_H$, then by Corollary~\ref{cor:big-h-energy}, we have $E_{\Delta}(G) > E_{\Delta}(H)$. This completes the proof.
\end{proof}

\begin{lem}\label{lem:-ve-ev}
Let $s \geq 2$ and $h\geq 1$. Let $n_i \geq 2$ for $1\leq i \leq s$ be positive integers and $n=h+\sum_{i=1}^s n_i$. Let $H=K_{n_1,\cdots,n_p,\cdots,n_q,\cdots,n_s, \scriptsize{\underbrace{1,1,\ldots,1}_{h\ times}}}$ and $H^*= K_{n_1,\cdots,n_p+1,\cdots,n_q-1,\cdots,n_s, \scriptsize{\underbrace{1,1,\ldots,1}_{h\ times}}}$. Let $\lambda_1(H)\geq \lambda_2(H)\geq \ldots \geq  \lambda_n(H)$ and $\lambda_1(H^*)\geq \lambda_2(H^*)\geq \ldots \geq  \lambda_n(H^*)$ be the eigenvalues of the squared distance matrix $\Delta(H)$ and $\Delta(H^*)$, respectively.  If $\lambda_{s+1}(H) \leq 0$, then 
$-1<\lambda_{s+1}(H^*) < \lambda_{s+1}(H) \leq 0.$ 
\end{lem}
\begin{proof}
For $\lambda_{s+1}(H) \leq 0 $, using  part $(ii)$ of Lemma~\ref{lem:eigen-sq-dist}, we have 
\begin{equation}\label{eqn:4.8}
 \sum_{i=1}^s \frac{n_i}{3n_i-4} \geq h-1.
\end{equation}
For $x=0$ and suitable choices in Eqn.~\eqref{eqn_energy_ineq2}, we have the following inequality
$$\frac{n_p+1}{3(n_p+1)-4} + \frac{n_q-1}{3(n_q-1)-4}>\frac{n_p}{3n_p-4} + \frac{n_q}{3n_q-4},$$
and hence by Eqn.~\eqref{eqn:4.8}, we have
$$\frac{n_p+1}{3(n_p+1)-4} + \frac{n_q-1}{3(n_q-1)-4} +\sum_{i=1 \atop i\neq p,q}^s \frac{n_i}{3n_i-4} > \sum_{i=1}^s \frac{n_i}{3n_i-4} \geq h-1.$$
Therefore, by  part $(ii)$ of Lemma~\ref{lem:eigen-sq-dist}, we have 
 $-1<\lambda_{s+1}(H^*) <  0.$ Further,  using Lemma~\ref{lem:pert-negative}, we have $-1<\lambda_{s+1}(H^*) < \lambda_{s+1}(H) \leq 0$ and $E_{\Delta}(H^*)> E_{\Delta}(H).$
\end{proof}

\begin{lem}\label{lem:sqen-mojor}
Let $s \geq 2$ and $h\geq 1$. Let $n_i \geq 2$, $m_i\geq 2$ for $1\leq i \leq s$ be a positive integers and $ (m_1,m_2,\cdots,m_s)\succ (n_1,n_2,\cdots,n_s)$. Let $G=K_{m_1,m_2,\cdots,m_s,\scriptsize{\underbrace{1,1,\ldots,1}_{h\ times}}}$ and $H=K_{n_1,n_2,\cdots,n_s, \scriptsize{\underbrace{1,1,\ldots,1}_{h\ times}}}$ be complete $t$-partite graphs on  $n=h+\sum_{i=1}^s m_i=h+\sum_{i=1}^s n_i$ vertices. Let  $\lambda_1(G)\geq \lambda_2(G)\geq \cdots \geq \lambda_n(G) $ and 
 $\lambda_1(H)\geq \lambda_2(H)\geq \ldots \geq  \lambda_n(H)$  be the eigenvalues of the squared distance matrix $\Delta(G)$ and $\Delta(H)$, respectively. If $\lambda_{s+1}(H) \leq 0$, then 
$-1<\lambda_{s+1}(G)  < \lambda_{s+1}(H) \leq 0$ and  $E_{\Delta}(G)> E_{\Delta}(H).$
\end{lem}

\begin{proof} 
Without loss of generality, let us assume that $n_1\geq n_2 \geq \cdots  \geq n_s \geq 2$ and  $m_1\geq m_2 \geq \cdots  \geq m_s \geq 2$. Let  $Y=(m_1,m_2,\cdots,m_s)$ and $X=(n_1,n_2,\cdots,n_s) $. Then  $X, Y \in \mathcal{B}_s$ and $ X \succ Y$. By Theorem~\ref{thm:maj_seq}, there exists a sequence in $\mathcal{B}_s$ such as $Y = Y_0 \succ Y_1 \succ \cdots \succ Y_l = X$ such that $ Y_i = Y_{i-1} + E_{j_i,k_i}^s$ for some $k_i > j_i, \; 1\leq i \leq l$. Thus, using Lemma~\ref{lem:-ve-ev}  inductively on the above sequence yields the result.
\end{proof}

For fixed values of $n,t$ and $h$,  we define two graphs  in $\mathcal{M}(n,t,h)$  which are similar to complete split graph and Tur{\'a}n graph  as follows:
\begin{enumerate}
\item $S_{n,t,h}\approxeq K_{n-2(t-1)+h, \scriptsize{\underbrace{2,2,\ldots,2}_{t-h-1\ times}}, \scriptsize{\underbrace{1,1,\ldots,1}_{h\ times} }}.$

\item $T_{n,t,h} \approxeq K_{\ceil*{n-h/t-h}, \cdots, \ceil*{n-h/t-h},\floor*{n-h/t-h},\cdots, \floor*{n-h/t-h}, \scriptsize{\underbrace{1,1,\ldots,1}_{h\ times}}}$
\end{enumerate}

\begin{theorem}\label{thm:extrema-n-t-h}
Let $s \geq 2$ and $h\geq 1$. Let $t=s+h$ and  $n_i \geq 2$ for $1\leq i \leq s$ be positive integers.  Let $K_{n_1,n_2,\cdots,n_s,\scriptsize{\underbrace{1,1,\ldots,1}_{h\ times}}}$ be a complete $t$-partite graph on $n=h+\sum_{i=1}^s n_i$ vertices. Then
$$E_{\Delta}(T_{n,t,h}) \leq E_{\Delta}(K_{n_1,n_2,\cdots,n_s,\scriptsize{\underbrace{1,1,\ldots,1}_{h\ times}}}) \leq E_{\Delta}(S_{n,t,h}).$$
Moreover, in this case if $\lambda_{s+1}(T_{n,t,h})\leq 0$, then the maximal graph $S_{n,t,h}$ and the minimal graph $T_{n,t,h}$ are unique upto  isomorphism.
\end{theorem}

\begin{proof}
By Theorems~\ref{thm:maj_tuple} and~\ref{thm:thm_compare} and Corollary~\ref{lem:major222}, we have
\begin{equation}\label{eqn:4.7-mojor}
\begin{cases}
(n-2(t-1)+h, \underbrace{2,2,\ldots,2}_{t-h-1\ times}, \underbrace{1,1,\ldots,1}_{h\ times})\succeq (n_1,n_2,\cdots,n_s,\underbrace{1,1,\ldots,1}_{h\ times}) \mbox{ and }\\
\\
(n_1,n_2,\cdots,n_s,\underbrace{1,1,\ldots,1}_{h\ times})\succeq 
\Big(\ceil*{\frac{n-h}{t-h}}, \cdots, \ceil*{\frac{n-h}{t-h}},\floor*{\frac{n-h}{t-h}},\cdots, \floor*{\frac{n-h}{t-h}}, \underbrace{1,1,\ldots,1}_{h\ times}\Big).
\end{cases}
\end{equation}

In view of Eqn.~\eqref{eqn:4.7-mojor}, using part $(i)$ of Theorem~\ref{thm:energy-mojor}, we have 
 $$E_{\Delta}(T_{n,t,h}) \leq E_{\Delta}(K_{n_1,n_2,\cdots,n_s,\scriptsize{\underbrace{1,1,\ldots,1}_{h\ times}}}) \leq E_{\Delta}(S_{n,t,h}).$$
Further, under the assumption  $\lambda_{s+1}(T_{n,t,h})\leq 0$, the uniqueness of extremal graphs $T_{n,t,h}$ and $S_{n,t,h}$ follows from Lemma~\ref{lem:sqen-mojor}. 
\end{proof}

\begin{rem}\label{rem:n-t-h}
For fixed values of $n,t$ and $h$, recall that $\mathcal{M}(n,t,h)$ is the class of complete $t$-partite graphs on $n$ vertices which has exactly $h$ number of partitions  of order $1$, {\it i.e,} 
\begin{align*}
\mathcal{M}(n,t,h)&=\left\{K_{n_1,n_2,\cdots,n_t} \; : \; n=\sum_{i=1}^t n_i \mbox{ and }h=|\{ i: n_i=1\}| \right\}\\
&=\left\{K_{n_1,n_2,\cdots,n_s,\scriptsize{\underbrace{1,1,\ldots,1}_{h\ times}}} \; : \;  t=s+h,\  n_i\geq 2 \mbox{ for } 1\leq i \leq s  \mbox{ and } n=h+\sum_{i=1}^s n_i  \right\}.
\end{align*}
Thus,  every complete $t$-partite graphs in $\mathcal{M}(n,t,h)$ is uniquely identified (upto isomorphism) with an $s$-tuple $(n_1,n_2,\cdots,n_s)$ of positive integers such that $t=s+h$, $n-h=\sum_{i=1}^s n_i$ and $n_1\geq n_2 \geq \cdots  \geq n_s \geq 2.$

Let $G$ be a complete $t$-partite graphs in $\mathcal{M}(n,t,h)$ is uniquely identified with  $(n_1,n_2,\cdots,n_s)$. Let  $\Delta(G)$ be the  squared distance matrix of $G$ and $\lambda_1(G)\geq \lambda_2(G)\geq \cdots \geq\lambda_n(G)$ be the eigenvalues of  $\Delta(G)$. Then, by part $(ii)$ of Lemma~\ref{lem:eigen-sq-dist}, we have    $\lambda_{s+1}(G)>0  \textup{ if and only if  } h-1 > \sum_{i=1}^s \dfrac{n_i}{3n_i-4},$ $\lambda_{s+1}(G)=0  \textup{ if and only if  } h-1 = \sum_{i=1}^s \dfrac{n_i}{3n_i-4}$  and  $\lambda_{s+1}(G)<0  \textup{ if and only if  } h-1 < \sum_{i=1}^s \dfrac{n_i}{3n_i-4}.$ 
\end{rem}

 From the proof of Theorems~\ref{thm:energy-mojor} and~\ref{thm:extrema-n-t-h}, we can  conclude the following: 
For fixed values of $n,t$ and $h$,   there exists a sequence of complete $t$-partite graphs $ \{ G_i : i=0,1,\ldots,l \mbox{ with } G_0=S_{n,t,h} \mbox{ and } G_l =T_{n,t,h} \}$  in $\mathcal{M}(n,t,h)$ such that 
$$E_{\Delta}(S_{n,t,h})=E_{\Delta}(G_0)\geq E_{\Delta}(G_1)\geq E_{\Delta}(G_2)\geq \cdots \geq E_{\Delta}(G_l)=E_{\Delta}(T_{n,t,h}).$$
Furthermore, by  Remark~\ref{rem:n-t-h} and Theorem~\ref{thm:maj_seq}, each $G_i$ is uniquely determine by $s$-tuple $Y_i$ and $ Y_i = Y_{i-1} + E_{j_i,k_i}^s$ for some  $k_i > j_i, \; 1\leq i \leq l$ such that \\
 \hspace*{.5cm}$(n-2(t-1)+h, \underbrace{2,2,\ldots,2}_{t-h-1\ times}) = Y_0 \succ Y_1 \succ \cdots \succ Y_l =\Big( \ceil*{\frac{n-h}{t-h}}, \cdots, \ceil*{\frac{n-h}{t-h}},\floor*{\frac{n-h}{t-h}},\cdots, \floor*{\frac{n-h}{t-h}} \Big).$

Using the above notations, we provide a few examples that show in general the maximal and the minimal value of the squared distance energy of graphs in $\mathcal{M}(n,t,h)$ is not uniquely attained by  $S_{n,t,h} \mbox{ and } T_{n,t,h}$, respectively.

\begin{ex} 
Let $n=31, t=15$ and $h=7$. In this case $s=8$ and by Remark~\ref{rem:n-t-h}, every complete $15$-partite graphs in  $\mathcal{M}(31,15,7)$ is uniquely identified with $8$-tuple. Then $S_{31,15,7}$
is identified with $(10,2,2,2,2,2,2,2)$ and $T_{31,15,7}$ is identified with $(3,3,3,3,3,3,3,3)$. Let us consider the chain
\begin{align*}
&(10,2,2,2,2,2,2,2)\succ(9,3,2,2,2,2,2,2)\succ(8,4,2,2,2,2,2,2)\succ(7,4,3,2,2,2,2,2)\\
\succ &(6,4,4,2,2,2,2,2) \succ(5,4,4,3,2,2,2,2) \succ(4,4,4,4,2,2,2,2)\succ(4,4,4,3,3,2,2,2)\\
\succ &(4,4,3,3,3,3,2,2)\succ (4,3,3,3,3,3,3,2)\succ (3,3,3,3,3,3,3,3).
\end{align*}
The corresponding sequence of complete $15$-partite graphs  in $\mathcal{M}(31,15,7)$ is given by $$ \{ G_i : i=0,1,\ldots,10 \mbox{ with  } G_0=S_{31,15,7} \mbox{ and } G_{10} =T_{31,15,7} \}.$$
In this case $\lambda_{s+1}=\lambda_{9}$, and by Remark~\ref{rem:n-t-h}, we have $\lambda_{9}(G_i) < 0$ for $i=0,1,\ldots,5$, $\lambda_{9}(G_6) = 0$ and  $\lambda_{9}(G_i) > 0$ for $i=7,8,9,10$. Thus,  Theorem~\ref{thm:inertia-energy} yields that $E_{\Delta}(T_{31,15,7})=E_{\Delta}(G_i)=8(n-t)+2(h-1)=140$ for $i=7,8,9,10$ and hence the minimal value of the squared distance energy is not uniquely attained by $T_{31,15,7}$. Furthermore, since 
$\lambda_{9}(S_{31,15,7})=\lambda_{9}(G_0)<0$, so by Lemma~\ref{lem:pert-negative}, the  maximal value of the squared distance energy is uniquely attained by $S_{31,15,7}$.
\end{ex}
\begin{ex} 
Let $n=30, t=15$ and $h=7$. In this case $s=8$ and by Remark~\ref{rem:n-t-h},  $S_{30,15,7}$
is identified with $(9,2,2,2,2,2,2,2)$ and $T_{30,15,7}$ is identified with $(3,3,3,3,3,3,3,2)$. Let us consider the chain $(9,2,2,2,2,2,2,2)\succ(8,3,2,2,2,2,2,2)\succ(7,4,2,2,2,2,2,2)\succ(6,4,3,2,2,2,2,2)$ \\
\hspace*{1.3cm}$\succ (5,4,4,2,2,2,2,2) \succ(4,4,4,3,2,2,2,2)\succ(4,4,3,3,3,2,2,2)\succ(4,3,3,3,3,3,2,2)$\\
\hspace*{1.3cm}$\succ (3,3,3,3,3,3,3,2).$\\
The corresponding sequence of complete $15$-partite graphs  in $\mathcal{M}(30,15,7)$ is given by $$ \{ G_i : i=0,1,\ldots,8 \mbox{ with  } G_0=S_{30,15,7} \mbox{ and } G_{8} =T_{30,15,7} \}.$$
In this case $\lambda_{s+1}=\lambda_{9}$, and by  Remark~\ref{rem:n-t-h}, we have $\lambda_{9}(G_i) < 0$ for $i=0,1,\ldots,5$,  and  $\lambda_{9}(G_i) > 0$ for $i=6,7,8$. Thus,  Theorem~\ref{thm:inertia-energy} yields that $E_{\Delta}(T_{30,15,7})=E_{\Delta}(G_i)=132$ and hence the minimal value of the squared distance energy is not uniquely attained by $T_{30,15,7}$. Furthermore, since 
$\lambda_{9}(S_{30,15,7})=\lambda_{9}(G_0)<0$, so by Lemma~\ref{lem:pert-negative}, the  maximal value of the squared distance energy is uniquely attained by $S_{30,15,7}$.
\end{ex}

\begin{ex} 
Let $n=17, t=10$ and $h=6$. In this case $s=4$ and by Remark~\ref{rem:n-t-h},  $S_{17,10,6}$
is identified with $(5,2,2,2)$.  Since $\lambda_{s+1}(S_{17,10,6})=\lambda_{5}(S_{17,10,6})>0$, so by  Lemma~\ref{lem:pert-non-negative}, we have $E_{\Delta}(S_{17,10,6}) = E_{\Delta}(G)$ for every $G$ in $\mathcal{M}(17,10,6)$. Therefore, neither the maximal nor the minimal value of the squared distance energy for graphs in $\mathcal{M}(17,10,6)$ are not uniquely attained by  $S_{17,10,6} \mbox{ and } T_{17,10,6}$, respectively.
\end{ex}

We conclude this section with the following result  that finds the extremal graphs for maxima and minima of $E_{\Delta}(K_{n_1,n_2,\cdots,n_t})$. We first prove a lemma useful for determining the uniqueness of these extremal graphs.

\begin{lem}\label{lem:h-hg}
Let $t,n$ be positive integers and  $t\geq 2$ such that $t+2\leq n< 2t.$ Let $T_{n,t}\approxeq K_{m_1,m_2,\cdots,m_t}$ and   $h=|\{ i: m_i=1\}|$. Let $G=K_{n_1,n_2,\cdots,n_t}$  be a complete $t$-partite graph on $n=\sum_{i=1}^t n_i$ vertices and  $h_G=|\{i : n_i=1\}|$. If $G$ is not isomorphic to $T_{n,t}$, then $h< h_G.$
\end{lem}
\begin{proof}
Let $T_{n,t}\approxeq K_{m_1,m_2,\cdots,m_t}$ with $m_1 \geq m_2\geq \cdots \geq m_t$ and $n=\sum_{i=1}^t m_i$, then $m_t=\ds \floor*{\frac{n}{t}}=1$. Let $h=|\{ i: m_i=1\}|$ and $s=n-t$. Then $m_i= 2$ for $1\leq i\leq s$ and $m_i=1$ for $s+1\leq i\leq t=s+h$. 

Let  $G=K_{n_1,n_2,\cdots,n_t}$  with $n_1 \geq n_2\geq \cdots \geq n_t$. Let  $h_G=|\{i : n_i=1\}|$ and $s_G=t-h_G.$  Then $n_i\geq 2$ for $1\leq i\leq s_G$ and $n_i=1$ for $s_G+1\leq i\leq t=s_G+h_G$. Therefore, $n=h+2s =  h_G +\sum_{i=1}^{s_G} n_i $ and hence $ h+ 2(s-s_G) = h_G + \sum_{i=1}^{s_G}(n_i-2)$. Further, using $t=s+h=s_G+h_G$, we have $s-s_G=h_G-h$. Thus, $ h+ 2(h_G-h) = h_G + \sum_{i=1}^{s_G}(n_i-2)$, which implies that $h_G-h = \sum_{i=1}^{s_G}(n_i-2).$

If $G$ is not isomorphic to $T_{n,t}$, then  $n_k> 2$ for some $1\leq k\leq s_G$. Therefore,  $h_G-h = \sum_{i=1}^{s_G}(n_i-2)> 0$ and hence the result follows.
\end{proof}

\begin{theorem}
Let $t,n$ be positive integers and $t\geq 2.$ Let $K_{n_1,n_2,\cdots,n_t}$  be a complete $t$-partite graph on  $n=\sum_{i=1}^t n_i$ vertices. Then, 
$$E_{\Delta}(T_{n,t}) \leq E_{\Delta}(K_{n_1,n_2,\cdots,n_t}) \leq E_{\Delta}(S_{n,t}).$$
Moreover, in this case,  the maximal graph $S_{n,t}$ is  unique upto  isomorphism, while the minimal graph $T_{n,t}$ is  unique upto  isomorphism if and only if $n\leq 2t+1.$

\end{theorem}
\begin{proof}
By Theorem~\ref{thm:thm_compare}, we have $
(n-t+1, \underbrace{1,\cdots,1}_{t-1}) \succeq (n_1,\cdots,n_t) \succeq (\ceil*{\frac{n}{t}},\cdots,\ceil*{\frac{n}{t}},\floor*{\frac{n}{t}},\cdots,\floor*{\frac{n}{t}}),$ and hence using part ($ii$) of Theorem~\ref{thm:energy-mojor} yields that 
$E_{\Delta}(T_{n,t}) \leq E_{\Delta}(K_{n_1,n_2,\cdots,n_t}) \leq E_{\Delta}(S_{n,t}).$ Now we proceed to prove the uniqueness of the extremal graphs $S_{n,t}$ and $T_{n,t}$.

Let  $G=K_{n_1,n_2,\cdots,n_t}$  be a complete $t$-partite graph on $n$ vertices and   $h_G=|\{i : n_i=1\}|.$   If $G$ is not isomorphic to $S_{n,t}$, then $h_G < t-1$. Therefore,  Corollary~\ref{cor:big-h-energy} yields that $ E_{\Delta}(G) < E_{\Delta}(S_{n,t})$ and hence the squared distance energy  is attained maximum uniquely by $S_{n,t}$  upto isomorphism. Next, to prove the minimization part we consider the cases $n> 2t+1$ and $n\leq 2t+1$, separately.\\

\noindent \underline{\textbf{Case 1:}} Let $n> 2t+1$. Then $\ds \floor*{\frac{n}{t}} \geq 2$. Therefore, by Theorem~\ref{thm:ni-geq-2}, we have $E_{\Delta}(T_{n,t})=8(n-t)$ and hence  Corollary~\ref{cor:big-h-energy} yields that $E_{\Delta}(T_{n,t}) \leq E_{\Delta}(G)$, where  $G=K_{n_1,n_2,\cdots,n_t}$ with $n=\sum_{i=1}^t n_i$.   Further, in this case, let $T_{n,t}=K_{m_1,m_2,\cdots,m_t}$ with $m_1 \geq m_2\geq \cdots \geq m_t$ and $n=\sum_{i=1}^t m_i$. Then, $\ds m_1=m_2=\ceil*{\frac{n}{t}}>2$ and $\ds m_t=\floor*{\frac{n}{t}}\geq 2.$ If  $H=K_{m_1+1,m_2-1,\cdots,m_t}$, then $m_1+1,m_2-1,\cdots,m_t \geq 2$ and hence by Theorem~\ref{thm:ni-geq-2}, we have $E_{\Delta}(H)=8(n-t).$ Therefore, in this case, there exists atleast two non isomorphic graphs for which the minimum attained.\\

\noindent \underline{\textbf{Case 2:}} Let $n\leq  2t+1$. Then $t\leq n\leq 2t+1.$\\

Let $t\leq n<  2t$. For $n= t$,  the class of  complete $t$-partite graphs on $n$ vertices consisting of a single element $K_t$ and for  $n=t+1$, the class of  complete $t$-partite graphs on $n$ vertices consisting of a single element  $K_{2,1,1,\ldots,1}$ upto isomorphism. Thus, the result is vacuously true for $n= t$ and $n=t+1$. For $t+2\leq n<  2t$, if $T_{n,t}\approxeq K_{m_1,m_2,\cdots,m_t}$ with $m_1 \geq m_2\geq \cdots \geq m_t$ and $n=\sum_{i=1}^t m_i$, then $m_t=\ds \floor*{\frac{n}{t}}=1$. Let $h=\{ i: m_i=1\}$ and $s=n-t$. Then, $m_i\geq 2$ for $1\leq i\leq s$ and $m_i=1$ for $s+1\leq i\leq t=s+h$.  If  $G=K_{n_1,n_2,\cdots,n_t}$  is a complete $t$-partite graphs on $n$ vertices with $h_G=|\{i : n_i=1\}|$ and $G$ is not isomorphic to $T_{n,t}$, then by Lemma~\ref{lem:h-hg}, we have $h<h_G$ and using Corollary~\ref{cor:big-h-energy}, we get $E_{\Delta}(T_{n,t}) < E_{\Delta}(G)$. Hence, the result is  hold true for  $t+2\leq n<  2t$. 

Next, if $n=2t$, then $T_{n,t}=K_{2,2,\cdots,2}$. If  $G=K_{n_1,n_2,\cdots,n_t}$  is a complete $t$-partite graph on $n$ vertices with $h_G=|\{i : n_i=1\}|$ and $G\neq T_{n,t}$, then $h_G \geq 1$ and hence using Corollary~\ref{cor:big-h-energy}, we get $E_{\Delta}(T_{n,t}) < E_{\Delta}(G)$. 

Finally, if $n=2t+1$, then $T_{n,t}\approxeq K_{3,2,2,\cdots,2}$.  If  $G=K_{n_1,n_2,\cdots,n_t}$  is a complete $t$-partite graph on $n$ vertices with $h_G=|\{i : n_i=1\}|$ and $G$ is not isomorphic to $T_{n,t}$,  then $h_G \geq 1$, and hence using Corollary~\ref{cor:big-h-energy}, we get $E_{\Delta}(T_{n,t}) < E_{\Delta}(G)$. This completes the proof. 
\end{proof}

\section{Spectral Radius of $\Delta(K_{n_1,n_2,\cdots,n_t})$}\label{sec:sp-radius}

In this section, we will show  that $ (m_1,m_2,\cdots,m_t)\succ (n_1,n_2,\cdots,n_t)$ implies that  $\rho(K_{m_1,m_2,\cdots,m_t}) > \rho(K_{n_1,n_2,\cdots,n_t})$, and as consequence prove that $\rho(K_{n_1,n_2,\cdots,n_t})$  with $n=\sum_{i=1}^t n_i$ is uniquely attains  its maximum value at $S_{n,t}$ and minimum value at $T_{n,t}$.  We begin with a few lemmas that will be useful  to prove the main results of the section.

\begin{lem}\label{lem:lem_bound}
Let $t\geq 2$ and $n_1,n_2,\cdots,n_t$ be positive integers such that $n_1 \ge n_2 \ge \cdots \ge n_t$. Then $$\rho(K_{n_1,n_2,\cdots,n_t}) > 4(n_1-1).$$
\end{lem}
\begin{proof}
Since $\Delta(K_{n_1,n_2})$ is a principal submatrix of  $\Delta(K_{n_1,n_2,\cdots,n_t})$, so using
Theorem~\ref{thm:inter}, we have $\rho(K_{n_1,n_2,\cdots,n_t}) \geq \rho(K_{n_1,n_2}).$ By  Eqn.~\eqref{eqn:char_eqn}, the characteristic equation of  $\Delta(K_{n_1,n_2})$ is given by 
$$(x+4)^{n_1+n_2-2}\left[  \prod_{i=1}^2 (x+4-3n_i) -\sum_{i=1}^2 \left( n_i\prod_{j \neq i}(x+4-3n_j) \right) \right] = 0,$$
and hence 
\begin{align*}
\rho(K_{n_1,n_2}) &= 2(n_1+n_2) + \sqrt{4(n_1-n_2)^2+n_1n_2} -4\\
                  &>2(n_1+n_2) + \sqrt{4(n_1-n_2)^2} -4\\
                  &=4(n_1-1),
\end{align*}
and the result follows.
\end{proof}

\begin{lem}\label{lem:ch-max-root}
Let $G$ and $H$ be connected graphs. Let $P_{\Delta}(G,x)$ and $P_{\Delta}(H,x)$ be the characteristic polynomials of $\Delta(G)$ and $\Delta(H)$, respectively. Then the following results hold.
\begin{itemize}
\item [$(i)$] If $x\geq \rho(G),$ then $P_{\Delta}(G,x)>0.$

\item [$(ii)$] If $P_{\Delta}(H,x)> P_{\Delta}(G,x)$ for $x\geq \rho(G),$ then $\rho(H)< \rho(G).$
\end{itemize}
\end{lem}
\begin{proof}
Let $P(G,y)< 0$ for some $y\geq \rho(G)$. Since $P_{\Delta}(G,x)$ is a monic polynomial, there exists sufficiently large  $z$ such that $P_{\Delta}(G,z)>0$. Then, using Intermediate value theorem   $P_{\Delta}(G,x)$ have  a root $w$ in the interval $(y,z)$. This implies that  $w >\rho(G)$, which leads to a contradiction. This proves part $(i).$ 

Next, if $\rho(H)\geq \rho(G)$, then  by part $(i)$ we have $P_{\Delta}(G,\rho(H))>0$, but $P_{\Delta}(H,\rho(H))=0$,  which is a contradiction to the hypothesis $P_{\Delta}(H,x)> P_{\Delta}(G,x)$ for $x\geq \rho(G)$ and the result follows.
\end{proof}

\begin{lem}\label{lem:pert}
Let $t \ge 2$ and $n_1,n_2,\cdots, n_t$ be positive integers such that  $n_p \ge n_q+2$ for $1\leq p,q\leq t$. Then 
 $$\rho(K_{n_1,\cdots,n_p,\cdots,n_q,\cdots, n_t}) > \rho(K_{n_1,\cdots,n_p-1,\cdots,n_q+1,\cdots, n_t}).$$
\end{lem}
\begin{proof} Without loss of generality, let us assume $n_1 \ge n_2 \ge \cdots \ge n_t$. Let $G=K_{n_1,\cdots,n_p,\cdots,n_q,\cdots, n_t}$ and $H=K_{n_1,\cdots,n_p-1,\cdots,n_q+1,\cdots, n_t}.$  By Lemma~\ref{lem:lem_bound}, we have $\rho(G)> 4(n_1-1)$.  If $x \geq \rho(G)$, then using  $\rho(G)> 4(n_1-1)$, we have
\begin{equation}\label{eqn:inq1}
 x+4 -3n_i > 4(n_1-1)+4 -3n_s= 3(n_1-n_i)+n_1+4 > n_1 >0 \mbox{ for } 1\leq i \leq t.
\end{equation}
Using $n_p \ge n_q+2$, we get $(x+4-3(n_p-1))(x+4-3(n_q+1))-(x+4-3n_p)(x+4-3n_q)
= 9(n_p-(n_q+1))>0.$ 
Thus, using Eqn.~\eqref{eqn:inq1} for  $x \geq \rho(G)$, we have 
\begin{equation}\label{eqn:inq2}
(x+4-3(n_p-1))(x+4-3(n_q+1))>(x+4-3n_p)(x+4-3n_q)>0.
\end{equation}
Therefore,  for $x \geq \rho(G)$, we have $(x+4-3(n_p-1)),(x+4-3(n_j-1))$ and $(x+4 -3n_i)$ for $1\leq i \leq t $ are non-zero and hence the characteristic polynomials   $ P_{\Delta}(G,x)$ and $P_{\Delta}(H,x)$ can be written as
{\small
\begin{align*}
P_{\Delta}(G,x) &= (x+4)^{n-t} (x+4-3n_p)(x+4-3n_q)\\
&\quad \quad \times  \left(1 - \frac{n_i}{x+4-3n_p} - \frac{n_j}{x+4-3n_q} - \sum_{i=1\atop i\neq p,q}^t \frac{n_s}{x+4-3n_i}  \right) \prod_{i=1\atop i\neq p,q }^t (x+4-3n_i) \mbox{ and }\\
P_{\Delta}(H,x) &= (x+4)^{n-t} (x+4-3(n_p-1))(x+4-3(n_q+1))  \\
&\quad \quad \times  \left(1 - \frac{n_i-1}{x+4-3(n_p-1)} - \frac{n_j+1}{x+4-3(n_q+1)} - \sum_{i=1\atop i\neq p,q}^t \frac{n_s}{x+4-3n_i}  \right)\prod_{i=1\atop i\neq p,q}^t (x+4-3n_i), 
\end{align*}}
respectively. Further,  $n_p \geq n_q+2$ yields  $3(n_p-1)\geq 3(n_q+1)$ and hence $(x+4-  3(n_q+1)) \geq (x+4-3(n_p-1)).$ Similarly,  using $n_p \geq n_q+2$, we get  $(x+4-  3n_q) >(x+4-3n_p).$ 
By Eqns.~\eqref{eqn:inq1} and~\eqref{eqn:inq2},  for $x \geq \rho(G)$ we know that $(x+4-  3(n_q+1)),  (x+4-3(n_p-1)), (x+4-  3n_q),(x+4-3n_p)>0$ and hence for $x \geq \rho(G)$, we have
\begin{equation}\label{eqn:inq3}
(x+4-  3n_q)(x+4-  3(n_q+1)) > (x+4-3n_p)(x+4-3(n_p-1)).
\end{equation}
For  $a=(x+4-3n_p)$ and $b= (x+4-  3n_q) $, we have $3m+a=3n+b=x+4$. Thus, in view of Eqn~\eqref{eqn:inq3},  substituting $ a=(x+4-3n_p)$ and $b= (x+4-  3n_q) $ in part $(i)$ of Lemma~\ref{lem:ab-ineq}, we have 
\begin{equation}\label{eqn:inq4}
\frac{n_p}{x+4-3n_p} + \frac{n_q}{x+4-3n_q} > \frac{n_p-1}{x+4-3(n_p-1)} + \frac{n_q+1}{x+4-  3(n_j+1)} \ \mbox{ for } x \geq \rho(G).
\end{equation}
Let $$h(x)=1 - \frac{n_p}{x+4-3n_p} - \frac{n_q}{x+4-3n_q} - \sum_{i=1\atop i\neq p,q}^t \frac{n_i}{x+4-3n_i}.$$ We claim that $h(x)\geq 0$ for $ x \geq \rho(G)$. Suppose on the contrary, we have $h(y)<0$ for some $ y \geq \rho(G)$. Then,  Eqn.~\eqref{eqn:inq1}  yields that $P_{\Delta}(G,y) <0$, which is a  contradiction to part $(i)$ of Lemma~\ref{lem:ch-max-root}.  Thus, for $ x \geq \rho(G)$,  Eqn.~\eqref{eqn:inq4} yields that
\begin{align*}
& 1 - \frac{n_p-1}{x+4-3(n_p-1)} - \frac{n_q+1}{x+4-3(n_q+1)} - \sum_{i=1\atop i\neq p,q}^t \frac{n_i}{x+4-3n_i}\\
\geq & 1 - \frac{n_p}{x+4-3n_p} - \frac{n_q}{x+4-3n_q} - \sum_{i=1\atop i\neq p,q}^t \frac{n_i}{x+4-3n_i}=h(x) \geq 0.
\end{align*}
Therefore, using Eqns~\eqref{eqn:inq1} and~\eqref{eqn:inq2}, we have $ P_{\Delta}(H,x) > P_{\Delta}(G,x)$ for all $x\geq \rho(G)$ and hence by part $(ii)$ of Lemma~\ref{lem:ch-max-root}, the desired result follows.
\end{proof}

Now we are ready to prove the main result of the section that gives $S_{n,t}$ and  $T_{n,t}$ are the extremal graphs for which the spectral radius of the squared distance matrix of $K_{n_1,n_2,\cdots,n_t}$ with $n=\sum_{i=1}^t n_i$ vertices uniquely attains its maxima and minima, respectively.

\begin{theorem}
For $t\geq 2$, let  $n_1,n_2,\cdots,n_t$ and  $m_1,m_2,\cdots,m_t$  be positive integers. Then the following results hold.

\begin{enumerate}
\item[($i$)]  If $ (m_1,m_2,\cdots,m_t)\succ (n_1,n_2,\cdots,n_t)$, then $\rho(K_{m_1,m_2,\cdots,m_t}) > \rho(K_{n_1,n_2,\cdots,n_t}).$

\item[($ii$)] If  $n =  \sum_{i=1}^t n_i$, then $ \rho(T_{n,t}) \leq  \rho(K_{n_1,n_2,\cdots,n_t}) \leq \rho(S_{n,t}). $ Moreover,  the maximal graph $S_{n,t}$ and the minimal graph $T_{n,t}$ are unique upto  isomorphism.
\end{enumerate}
\end{theorem}
\begin{proof}
Without loss of generality, let us assume $n_1\geq n_2 \geq \cdots  \geq n_t $ and  $m_1\geq m_2 \geq \cdots  \geq m_t$. Let  $Y=(m_1,m_2,\cdots,m_t)$ and $X=(n_1,n_2,\cdots,n_t) $. Then  $X, Y \in \mathcal{B}_t$ and $ X \succ Y$. By Theorem~\ref{thm:maj_seq}, there exists a sequence in $\mathcal{B}_t$ such as $Y = Y_0 \succ Y_1 \succ \cdots \succ Y_l = X$ such that $ Y_i = Y_{i-1} + E_{j_i,k_i}^t$ for some $k_i > j_i, \; 1\leq i \leq l$. Inductively applying  Lemma~\ref{lem:pert} yields that  $\rho(K_{m_1,m_2,\cdots,m_t}) > \rho(K_{n_1,n_2,\cdots,n_t}).$ This proves part $(i)$.

Next, let $n =  \sum_{i=1}^t n_i$. By Theorem~\ref{thm:thm_compare}, we have 
$$
(n-t+1, \underbrace{1,\cdots,1}_{t-1}) \succeq (n_1,\cdots,n_t) \succeq \Big(\ceil*{\frac{n}{t}},\cdots,\ceil*{\frac{n}{t}},\floor*{\frac{n}{t}},\cdots,\floor*{\frac{n}{t}}\Big),
$$ 
and hence using part $(i)$ yields that $ \rho(T_{n,t}) \leq  \rho(K_{n_1,n_2,\cdots,n_t}) \leq \rho(S_{n,t}).$ 
Since the inequality in part~$(i)$ is strict,  the spectral radius $\rho(K_{n_1,n_2,\cdots,n_t})$ is maximum if $K_{n_1,n_2,\cdots,n_t}\approxeq  S_{n,t}$  and is minimum  if $K_{n_1,n_2,\cdots,n_t}\approxeq T_{n,t}.$ 
\end{proof}

Next, we show that the part $(ii)$ of the above theorem directly follows from  Lemma~\ref{lem:pert} and hence provide an alternative proof for the same.

\begin{theorem}
Let $t \ge 2$ and $n_1,n_2,\cdots,n_t$ be  positive integers and $n =  \sum_{i=1}^t n_i$. Then the following   hold.
\begin{enumerate}
\item[(i)] The spectral radius of  $\Delta(K_{n_1,n_2,\cdots,n_t})$ is maximal if $K_{n_1,n_2,\cdots,n_t}\approxeq  S_{n,t}.$

\item [(ii)] The spectral radius of  $\Delta(K_{n_1,n_2,\cdots,n_t})$ is minimal if $K_{n_1,n_2,\cdots,n_t}\approxeq T_{n,t}.$ 
\end{enumerate}
\end{theorem}
\begin{proof}
Let  $K_{m_1,m_2,\cdots,m_t}$  be a  complete $t$-partite graph with $n=\sum_{i=1}^t m_i$  such that
$$\rho(K_{m_1,m_2,\cdots,m_t})=\max_{n = n_1+\cdots+n_t}  \rho(K_{n_1,n_2,\cdots,n_t}).$$
If there exist $1\leq p,q\leq t$ such that $m_p\geq m_q\geq 2$, then $(m_p+1)-(m_q-1)\geq 2$ and  by Lemma~\ref{lem:pert}, we have
$$\rho(K_{m_1,\cdots,m_p, \cdots,m_q, \cdots,m_t})
< \rho(K_{m_1,\cdots,m_p+1, \cdots,m_q-1, \cdots,m_t}),$$ 
which is a contradiction to the choice of $K_{m_1,m_2,\cdots,m_t}$. Thus, all the $m_1,m_2,\cdots, m_t$ are equal to $1$, except for one which is equal to $n-t+1$ and hence $K_{m_1,m_2,\cdots,m_t}\approxeq  S_{n,t}.$  This proves  part $(i)$.

Next, let  $K_{m_1,m_2,\cdots,m_t}$  be a  complete $t$-partite graph with $n=\sum_{i=1}^t m_i$  such that
$$\rho(K_{m_1,m_2,\cdots,m_t})=\min_{n = n_1+\cdots+n_t}  \rho(K_{n_1,n_2,\cdots,n_t}).$$ 
Then, $|m_p-m_q|\leq 1$ for all $p\neq q$, otherwise there exist $1\leq p,q\leq t$ such that $m_p -m_q \geq 2$ and hence by Lemma~\ref{lem:pert}, we have
$$\rho(K_{m_1,\cdots,m_p,\cdots,m_q,\cdots, m_t}) > \rho(K_{m_1,\cdots,m_p-1,\cdots,m_q+1,\cdots, m_t}),$$
which is a contradiction to the choice of $K_{m_1,m_2,\cdots,m_t}$. Thus, the condition $|m_p-m_q|\leq 1$ for all  $p\neq q$ implies that for all $1\leq i \leq t$ we have $m_i$ equal to $\ceil*{\dfrac{n}{t}}$ or $\floor*{\dfrac{n}{t}}$ and hence $K_{m_1,m_2,\cdots,m_t}\approxeq  T_{n,t}.$ This completes the proof.
\end{proof}

\section{Conclusion}
In this article, we have studied the squared distance matrix of complete multipartite graphs. 
We define the squared distance energy to be the sum of the modulus of the eigenvalues of the squared distance matrix. We determine the inertia of $\Delta(K_{n_1,n_2,\cdots,n_t})$ and using the result on inertia we compute the squared distance energy of $K_{n_1,n_2,\cdots,n_t}$. Furthermore, using majorization techniques we prove that for a fixed value of $n$ and $t$, both the spectral radius of the squared distance matrix and the squared distance energy of complete $t$-partite graphs on $n$ vertices are maximal for complete split graph $S_{n,t}$ and minimal for Tur{\'a}n graph $T_{n,t}$. We also discuss the uniqueness of these extremal graphs.

\section{Declarations}
\noindent{\textbf{\large Funding}}: Sumit Mohanty would like to thank the Department of Science and Technology, India, for financial support through the projects MATRICS (MTR/2017/000458).\\

\noindent{\textbf{\large Conflicts of interest/Competing interests}}: Not applicable.\\

\noindent{\textbf{\large Availability of data and material}}: Not applicable.\\

\noindent{\textbf{\large Code availability}}: Not applicable.\\

\noindent{\textbf{\large Authors' contributions}}: Both the authors have equal contribution.\\

\noindent{\textbf{\large Ethics approval}}: Not applicable.\\

\noindent{\textbf{\large Consent to participate}}: Not applicable.\\

\noindent{\textbf{\large Consent for publication}}: Not applicable.

\small{

}

\end{document}